\documentclass[a4paper,11pt]{article}

\usepackage{color}

\usepackage{ifthen}
\usepackage[utf8]{inputenc}
\usepackage[english]{babel}
\usepackage{amsmath,amsfonts,amssymb,amsthm}
\usepackage{centernot}
\usepackage{graphicx}
\usepackage{bbding}
\usepackage{xcolor}
\usepackage{mathtools}
\usepackage{hyperref}
\usepackage{enumerate}
\usepackage[norelsize]{algorithm2e}
\usepackage{hyperref}
\usepackage{tikz}
\usepackage{wasysym}

\newtheorem{theorem}{Theorem}
\newtheorem*{metathm}{Meta-Theorem}
\newtheorem{lemma}[theorem]{Lemma}
\newtheorem{claim}[theorem]{Claim}

\newtheorem{corollary}[theorem]{Corollary}

\theoremstyle{definition}
\newtheorem{definition}[theorem]{Definition}

\theoremstyle{remark}

\newcommand{\ignore}[1]{} 

\newcommand{\EE}{\ensuremath{\mathbb E}}

\newcommand{\PP}{\ensuremath{\mathbb P}}

\renewcommand{\phi}{\varphi}

\newcommand{\abs}[1]{\lvert#1\rvert}

\renewcommand{\P}{\mathrm{P}}
\newcommand{\polylog}{\mathrm{polylog}}

\newcommand{\Gnp}{G(n, p)}
\newcommand{\Hknp}{G^{(\ell)}(n, p)}

\newcommand{\dmax}{d_{\text{max}}}
\newcommand{\fo}{\textup{fo}}
\newcommand{\reg}{\textup{reg}}

\newcommand{\binoms}[2]{{\textstyle\binom{#1}{#2}}} 
\renewcommand{\phi}{\varphi}

\renewcommand{\P}{\PP}

\makeatletter
\def\slashedarrowfill@#1#2#3#4#5{%
  $\m@th\thickmuskip0mu\medmuskip\thickmuskip\thinmuskip\thickmuskip
   \relax#5#1\mkern-7mu%
   \cleaders\hbox{$#5\mkern-2mu#2\mkern-2mu$}\hfill
   \mathclap{#3}\mathclap{#2}%
   \cleaders\hbox{$#5\mkern-2mu#2\mkern-2mu$}\hfill
   \mkern-7mu#4$%
}

\def\rightslashedarrowfilla@{%
  \slashedarrowfill@\relbar\relbar{\raisebox{1.2pt}{$\scriptscriptstyle\diagup$}}\rightarrow}
\newcommand\xslashedrightarrowa[2][]{%
  \ext@arrow 0055{\rightslashedarrowfilla@}{#1}{#2}}

\makeatother

\newcommand{\ramcolk[1]}{\xrightarrow[{\text{\raisebox{3.5pt}[0pt]{{\tiny $#1$}}}}]{\text{{\tiny $ram$}}}}
\newcommand{\nramcolk[1]}{\xslashedrightarrowa[{\text{\raisebox{3.5pt}[0pt]{{\tiny $#1$}}}}]{\text{{\tiny $ram$}}}}
\newcommand{\aramcolk[1]}{\xrightarrow[{\text{\raisebox{3.5pt}[0pt]{{\tiny $#1$}}}}]{\text{{\tiny $a$-$ram$}}}}
\newcommand{\naramcolk[1]}{\xslashedrightarrowa[{\text{\raisebox{3.5pt}[0pt]{{\tiny $#1$}}}}]{\text{{\tiny $a$-$ram$}}}}
\newcommand{\aramcolp}{\xrightarrow[{\text{\raisebox{3pt}[0pt]{{\tiny $prp$}}}}]{\text{{\tiny $a$-$ram$}}}}
\newcommand{\naramcolp}{\xslashedrightarrowa[{\text{\raisebox{3pt}[0pt]{{\tiny $prp$}}}}]{\text{{\tiny $a$-$ram$}}}}
\newcommand{\gamecol}{\xrightarrow{\text{{\tiny $game$}}}}
\newcommand{\anycol}{\xrightarrow{\ *\ }}   
\newcommand{\nanycol}{\xslashedrightarrowa{\ *\ }}

\usepackage[margin=1in]{geometry}
\usepackage{verbatim}

\begin{document}




\thispagestyle{empty} 
\begin{center}

\LARGE An algorithmic framework for obtaining lower bounds\\
 for random Ramsey problems
\vspace{8mm}

\large{
\begin{tabular}{ccccccc}
Rajko Nenadov$^1$ & \quad & Yury Person$^2$ & \quad & Nemanja \v Skori\'c$^1$ & \quad & Angelika Steger$^1$ \\
{\small{rnenadov@inf.ethz.ch}} & \quad & {\small{person@math.uni-frankfurt.de}} & \quad & {\small{nskoric@inf.ethz.ch}}&\quad & {\small{steger@inf.ethz.ch}} 
\end{tabular}
}
\vspace{5mm}

\large
  $^1$Department of Computer Science \\
  ETH Zurich, 8092 Zurich, Switzerland
\vspace{4mm}

\large
  $^2$Institute of Mathematics\\
  Goethe-Universit\"at, 
  60325 Frankfurt am Main, Germany
\vspace{8mm}

\end{center}

\begin{abstract}
In this paper we introduce a general framework for proving lower bounds for various Ramsey type problems within random settings. The main idea is to view the problem from an algorithmic perspective: we aim at providing an algorithm that finds the desired colouring with high probability. Our framework allows to reduce the probabilistic problem of whether the Ramsey property at hand holds for random (hyper)graphs with edge probability $p$ to a deterministic question of whether there exists a finite graph that forms an obstruction. 

In the second part of the paper we apply this framework to address and solve various open problems. In particular, we extend the result of Bohman, Frieze, Pikhurko and Smyth (2010) for bounded anti-Ramsey problems in random graphs to the case of $2$ colors and to hypergraph cliques. As a corollary, this proves a matching lower bound for the result of Friedgut, R\"odl and Schacht (2010)  and, independently, Conlon and Gowers (2014+) for the classical Ramsey problem for hypergraphs in the case of cliques. Finally, we provide matching lower bounds for a proper-colouring version of anti-Ramsey problems introduced by Kohayakawa, Konstadinidis and Mota~(2014) in the case of cliques and cycles.
\end{abstract}

\section{Introduction and Results} \label{sec:intro}

A hypergraph $G$ is \emph{Ramsey} for a hypergraph $F$ and an integer $r$, if every colouring of the edges of $G$ with $r$ colours contains a copy of $F$ with all its edges having the same colour. A celebrated theorem of Ramsey~\cite{ramsey1930problem}  states that if $G$ is a large enough complete hypergraph then $G$ is Ramsey for $F$ and $r$. A priori it is not clear whether this follows from the density of a complete hypergraph or its rich structure. It was shown only later that actually the latter is the case: there exist sparse graphs with rich enough structure so that they are Ramsey for $F$. For example, a result of Ne\v set\v ril and R\"odl~\cite{nevsetvril1976ramsey} states that for every $k$ there exists a sparse graph $G$ that does not contain a clique of size $k+1$, but that nevertheless is Ramsey for a clique of size $k$. Nowadays, the easiest way to prove such result is by studying Ramsey properties of random (hyper)graphs. 

Over the last decades the study of various Ramsey-type problems for random (hyper)graphs received a lot of attention. In their landmark result, R\"odl and Ruci\'nski~\cite{RR93,RR94,Rodl:1995} gave a precise characterization of all edge probabilities $p = p(n)$ for which Ramsey's theorem holds in 
the random graph $\Gnp$ for a given graph $F$ and $r$ colors. 
 The corresponding problem for hypergraphs remained open for more than 15 years. Only recently, Friedgut, R\"odl and Schacht~\cite{friedgut2010ramsey}  and independently Conlon and Gowers~\cite{conlon2010combinatorial}  obtained an upper bound  analogous to the graph case. However, the question whether there exists a matching lower bound remained open.
 
More recently, other variations on  Ramsey-type problems in random graphs have been 
investigated. These are so-called anti-Ramsey properties such as finding 
rainbow copies of a given graph $F$ in any $r$-bounded colouring of $\Gnp$, initiated by Bohman, Frieze, Pikhurko and Smyth~\cite{Bohman:2010},  and 
in any proper edge-colouring of $\Gnp$, introduced by Kohayakawa, Konstadinidis and Mota~\cite{Kohayakawa:2011,Kohayakawa:2014}. 


The aim of our paper is twofold. First we introduce a general framework for proving lower bounds for Ramsey-type problems for random hypergraphs. Roughly speaking, the framework allows to reduce the probabilistic problem
\begin{center}
\emph{Does the Ramsey property at hand hold for}\\
\emph{random (hyper)graphs with edge probability $p$ w.h.p.?}
\end{center}
to a deterministic question of whether there exists a (hyper)graph that forms an \emph{obstruction}, or more precisely
\begin{center}
\emph{Does there exist a (hyper)graph with density at most $d(F,r)$ on at most $v(F,r)$ vertices }\\
\emph{that does not have the given Ramsey property?}
\end{center}
In the second part of the paper we then apply this framework to various Ramsey-type problems in random (hyper)graphs by 
 providing proofs of lower bounds that match the known upper bounds up to a constant factor.

\subsection{Definitions and Notations}
For background on graph theory we refer the reader to standard text books, see e.g.~\cite{Bollobas_book}. 
In particular, we denote the number of vertices and edges of a graph $G = (V, E)$ with $v(G)$ and $e(G)$, respectively.  For a subset of vertices $V' \subseteq V$, we denote with $G[V']$ the subgraph of $G$ induced by the vertices in $V'$. Furthermore, for a subset of vertices $S \subseteq V$ we use the shorthand $G \setminus S$ to denote the subgraph $G[V \setminus S]$. Similarly, for $E' \subseteq E$ we write $G \setminus E'$ to denote the graph $(V, E \setminus E')$, and by $G[E']$ we mean a graph with the edge set $E'$ on the vertex set $\cup_{e\in E'}e$. Given a graph $G$ and a vertex $v \in V(G)$, we write $N_G(v)$ for the set of neighbours of $v$ in $G$, $\deg_G(v) := |N_G(v)|$ for its degree and $\delta(G) = \min_{v \in V(G)} \deg_G(v)$
 denotes the minimum degree of $G$. If the graph $G$ is clear from the context, we omit it in the subscript. For two graphs $G_1$ and $G_2$, we write $G_1 \cong G_2$ if they are isomorphic.

 An $\ell$-uniform hypergraph $G$, or $\ell$-graph for short,  is a pair $(V,E)$ with the vertex set $V$ and $E\subseteq\binom{V}{\ell}$ the set of (hyper)edges. We will use the same notation as above for hypergraphs. Furthermore, a $k$-set is a set of cardinality $k$.
 
 The classical Ramsey problem is the following. Given two $\ell$-graphs $F$ and $G$ and an integer~$r$, we write
$$G \ramcolk[r] F$$
if every edge colouring of $G$ with $r$ colours contains a monochromatic copy of $F$. Clearly, it is essential to restrict the number of colours. Otherwise, using a different colour for each edge in $G$ trivially avoids any monochromatic copy of $F$. Theorem of Ramsey~\cite{ramsey1930problem} states that for every $\ell$ and $r$ and 
every $\ell$-graph $F$ 
we have, for a large enough $n$, that
\[
K^{(\ell)}_n \ramcolk[r] F,
\]
where $K^{(\ell)}_n$ denotes the complete $\ell$-graph $\left([n],\binom{[n]}{\ell}\right)$ on $n$ vertices.

It is natural to study analogues of Ramsey's theorem in the random setting. More precisely, we consider a binomial random $\ell$-uniform hypergraph $\Hknp$ on $n$ vertices in which every subset of size $\ell$ forms an edge with probability $p$ independently. In the case $\ell=2$ (the graph case) we use $\Gnp$ instead of $G^{(2)}(n,p)$. Given a (hyper)graph property $\mathcal{P}$, we say that a function $p_0 = p_0(n)$ is a threshold for $\mathcal{P}$ if 
$$
\lim_{n \rightarrow \infty} \Pr[\Hknp \in \mathcal{P}] = \begin{cases}
1 \quad \text{if} &p \gg p_0(n),\\
0 \quad \text{if} &p \ll p_0(n).
\end{cases}
$$
We say that an event $\mathcal{E}$ holds with high probability (w.h.p.\ for short) if $\lim_{n \rightarrow \infty} \Pr[\mathcal{E}] = 1$. It is easy to see that the Ramsey problem  induces a monotone property and  it follows from the result of Bollob{\'a}s and Thomason~\cite{bollobas1987threshold} that there has to exist some threshold $p_0(n)$. 

In this paper we will study $0$-statements of the above Ramsey-type problem and its variations for random $\ell$-graphs. Before giving an account on the previous and our results let us provide an intuition where the threshold for various Ramsey properties may be located (for most graphs $F$). Observe that the expected number of copies of $F$ in $\Hknp$ has the order of $n^{v(F)}p^{e(F)}$, where by $v(F)$ and $e(F)$ we denote the number of vertices and edges of $F$, respectively. On the other hand, the expected number of edges of $\Hknp$ is in the order of $n^\ell p$. That is, if $n^{v(F)}p^{e(F)} \ll n^\ell p$ then we expect the copies of $F$ to be loosely scattered -- and finding a colouring that avoids the desired copy of $F$ should be an easy task. Similarly, if $n^{v(F)}p^{e(F)} \gg n^\ell p$ we expect that the copies of $F$ overlap so heavily that any colouring should contain the desired copy $F$. 

Actually, the same argument holds for any subgraph of $F$ and this thus motivates the definition of the so-called \emph{$\ell$-density} that we now give. For an $\ell$-graph  $G=(V,E)$ on at least $\ell + 1$ vertices, we set $d_\ell(G) := (e(G) - 1)/(v(G)-\ell)$ and denote by $m_\ell(G)$ the maximum $\ell$-density of any subgraph of $G$, 
$m_\ell(G) = \max_{{J \subseteq G, v(J) \geq \ell + 1}} d_\ell(J)$.
If $m_\ell(G) = d_\ell(G)$, we say that $G$ is \emph{$\ell$-balanced}, and if in addition $m_\ell(G) > d_\ell(J)$ for every subgraph $J \subsetneq G$ with $v(J) \geq \ell + 1$, we say that $G$ is \emph{strictly $\ell$-balanced}. Another related notion which will be used extensively throughout the paper is the \emph{density} of an $\ell$-graph defined as $d(G) = e(G) / v(G)$. Similarly, we denote with $m(G)$ the maximum density over all subgraphs of $G$, i.e.
$m(G) = \max_{J \subseteq G} d(J)$.

\subsection{Results -- old and new}
\subsubsection{Ramsey's theorem for random $\ell$-graphs}
The systematic study of Ramsey properties of random graphs was initiated by 
\L{}uczak, Ruci\'nski and Voigt~\cite{LRV92} in the early nineties. Shortly thereafter  
 R\"odl and Ruci\'nski determined the threshold function 
of the graph Ramsey property  
for all graphs $F$. Below we state their result for all but a very special class of acyclic graphs. 

\begin{theorem}[\cite{RR93,RR94,Rodl:1995}] \label{thm:rr}
Let  $H$ be a graph that is not a forest of stars and, if $r = 2$, paths of length 3. Then there exist constants $c, C > 0$ such that
  \begin{equation*}
     \lim_{n\to \infty}\Pr[\Gnp \ramcolk[r] H] =
       \begin{cases}           
         1,&\text{if \(p \geq Cn^{-1/m_2(H)}\)}, \\
         0,&\text{if \(p \leq cn^{-1/m_2(H)}\)}.
       \end{cases}
  \end{equation*}
\end{theorem}
In the case when $F$ is a triangle, Friedgut, R\"odl, Ruci\'nski and Tetali~\cite{friedgut_sharp} have strengthened Theorem \ref{thm:rr}  by showing that there exists a sharp threshold.
 
 Extending Theorem \ref{thm:rr} to hypergraphs, R\"odl and Ruci\'nski~\cite{rodl1998ramsey} proved that for the $3$-uniform clique on $4$ vertices and $2$ colours the $1$-statement is determined by the $3$-density, as one would expect. They also conjectured that, similarly to the graph case, the threshold should be determined  by the $\ell$-density  for ``most'' of the  $\ell$-graphs $F$. R\"odl, Ruci\'nski and Schacht~\cite{rodl2007ramsey} later showed that the $1$-statement actually holds for all $\ell$-partite $\ell$-graphs. In full generality the $1$-statement was  resolved only recently by Friedgut, R\"{o}dl and Schacht~\cite{friedgut2010ramsey} and independently by Conlon and Gowers~\cite{conlon2010combinatorial}. 
\begin{theorem}[\cite{friedgut2010ramsey,conlon2010combinatorial}]\label{thm:hypergraph-1-statement}
  Let $F$ be an $\ell$-graph with maximum degree at least 2 and let $r \geq 2$. Then there exists a constant $C > 0$ such that for $p \geq C n^{-1/m_\ell(F)}$ we have
  \begin{equation*}
    \lim_{n\to \infty} \P[\Hknp \ramcolk[r] F] = 1.
  \end{equation*}
\end{theorem}

Recall, that  $K_k^{(\ell)}$ denotes a complete $\ell$-graph on $k$ vertices.
In this paper we make progress towards providing the missing lower bounds by resolving the case of cliques. 

\begin{theorem} \label{thm:ramsey_cliques}
Let $k, \ell$ be such that $2 \le \ell < k$ and let $r \ge 2$. Then there exist constants $c, C > 0$ such that
  \begin{equation*}
     \lim_{n\to \infty}\Pr[\Hknp \ramcolk[r] K_k^{(\ell)}] =
       \begin{cases}           
         1,&\text{if \(p \geq Cn^{-1/m_\ell(K_k^{(\ell)})}\)}, \\
         0,&\text{if \(p \leq cn^{-1/m_\ell(K_k^{(\ell)})}\)}.
       \end{cases}
  \end{equation*}
\end{theorem}

We will deduce Theorem~\ref{thm:ramsey_cliques} as a straightforward corollary from the results in the next subsection. We note that~\cite{PhDhenning} also contains a proof of Theorem~\ref{thm:ramsey_cliques} by different means.

\subsubsection{Anti-Ramsey property for $r$-bounded colourings}
If we allow colourings with an unbounded number of colours we arrive at the so-called anti-Ramsey problem where we are interested in finding a rainbow copy of $F$, i.e., a copy of $F$ in which each edge uses a different colour. 
Again, to avoid trivialities one needs to forbid colourings with too few colours. This has been done in several different ways. Here we insist that each colour is used at most $r$ times (we call this an $r$-bounded colouring).  
We write
$$G \aramcolk[r] F$$
if every $r$-bounded edge colouring of $G$ contains a rainbow copy of $F$.

Lefmann, R\"odl and Wysocka \cite{Lefmann:1996} considered the following question. Given a complete graph $G$ with edges colored using an $r$-bounded coloring, what is the largest $\ell$ such that $G$ contains a rainbow copy of $K_\ell$.
Bohman, Frieze, Pikhurko and Smyth~\cite{Bohman:2010} initiated the study of a similar question in $\Gnp$. The authors proved that given a graph $F$ and a constant $r \ge r(F)$, the threshold for the property of being $r$-bounded anti-Ramsey matches the intuition.

\begin{theorem}[\cite{Bohman:2010}] \label{thm:bohman_ar}
Let $F$ be a graph which contains a cycle. Then there exists a constant $r_0 = r_0(F)$ such that for each $r \ge r_0(F)$ there exist constants $c, C > 0$ and
\begin{equation*}
     \lim_{n\to \infty}\Pr[\Gnp \aramcolk[r] F] =
       \begin{cases}           
         1,&\text{if \(p \geq Cn^{-1/m_2(F)}\)}, \\
         0,&\text{if \(p \leq cn^{-1/m_2(F)}\)}.
       \end{cases}
  \end{equation*}
\end{theorem}

It is easy to see that for the case $F = K_3$ and $2$-bounded colourings there exists an obstruction, namely the complete graph on $4$ vertices. We refer the reader to \cite{Bohman:2010} for details regarding the results in the case $F = K_3$. For other graphs $F$ it is not obvious whether the restriction on $r$ is really needed. Indeed, the following theorem strengthens the 0-statement of Theorem \ref{thm:bohman_ar} by showing that $r=2$ actually suffices for most cases. In part $(iii)$ we also provide an extension to hypergraphs in the case of cliques.
 
\begin{theorem}\label{thm:2_bnd}
Let $\ell \ge 2$ and $F$ be an $\ell$-graph. Let $F'  \subseteq F$ be a strictly $\ell$-balanced subgraph such that $m_\ell(F') = m_\ell(F)$ . Then there exists a constant $c > 0$ such that $G \sim \Hknp$ w.h.p.\ satisfies $G \naramcolk[2] F$ if one of the following holds,
\begin{enumerate}[(i)]
\item $\ell = 2$, $F'$ contains a cycle, $F' \ncong \{K_3, C_4\}$ and $p \le cn^{-1/m_2(F)}$, or
\item $\ell=2$, $F' \cong C_4$ and $p \ll n^{-1/m_2(C_4)}$, or
\item $\ell \ge 3$, $r \ge \ell + 1$ and $(\ell, r) \neq (3, 4)$, $F' \cong K_r^{(\ell)}$ and $p \le cn^{-1/m_\ell(K_r^{(\ell)})}$.
\end{enumerate}
\end{theorem}

 
As an interesting corollary of Theorem \ref{thm:2_bnd}, we briefly mention the question of Maker-Breaker $F$-games on random (hyper)graphs. 
We write
$$G \gamecol F$$
if in the following game Maker has a winning strategy: two players,
Maker and Breaker, alternately claim unclaimed edges of $G$ until all the edges are claimed. Maker wins if he claims all the edges of some copy of $F$; otherwise Breaker wins. (For the sake of definiteness we assume that Maker has the first move.)


It is easy to see that the property of not being $2$-bounded anti-Ramsey for $F$ is stronger than being a Breaker's win in the Maker-Breaker $F$-game. Indeed, assume that a hypergraph $G$ is such that $G \naramcolk[2] F$. Then  Breaker can apply the following strategy: fix some $2$-bounded colouring of $G$ without a rainbow copy of $F$ and whenever Maker claims an edge, claim the other edge with the same colour. Then  Maker's graph corresponds to a rainbow subgraph of $G$ and thus does not contain an $F$-copy. Therefore, Theorem \ref{thm:2_bnd} slightly extends the result of Nenadov, Steger and Stojakovi\'c \cite{NenadovSteger:2014} by also providing a lower bound in the case of hypergraph cliques. 

\subsubsection{Anti-Ramsey property for proper edge colourings}
We write
$$G \aramcolp F$$
if every proper edge colouring of $G$ contains a rainbow copy of $F$. 

The first result on the relation between random graphs and the proper-colouring version of the anti-Ramsey property comes from the following question raised by Spencer: is it true that for every $g$ there exists a graph $G$ with girth at least $g$ such that $G \aramcolp C_\ell$ for some $\ell$. The question was answered in positive by R\"odl and Tuza~\cite{rodl1992rainbow}. They proved that for every $\ell$ there exists some sufficiently small $p = p(n)$ such that w.h.p.\ $\Gnp \aramcolp C_\ell$. Only much later, Kohayakawa, Kostadinidis and Mota~\cite{Kohayakawa:2011, Kohayakawa:2014} started a systematic study of this property in the random settings. In particular, they proved that the upper bound is as expected.

\begin{theorem}[\cite{Kohayakawa:2014}]
 Let $F$ be a graph. Then there exists a constant $C > 0$ such that for $p \ge Cn^{-1/m_2(F)}$ we have
 $$ \lim_{n \rightarrow \infty} \Pr[\Gnp \aramcolp F] = 1. $$
\end{theorem}
 
Note that $F = K_3$ is a trivial case since $K_3$ is an obvious obstruction. Therefore, any graph $F$ which contains $K_3$ as the $2$-densest subgraph is a potential candidate for having an obstruction. Indeed, the above authors showed in~\cite{Kohayakawa-Anti:2014} that there exists an infinite family of graphs for which the threshold is asymptotically below the guessed one. Here we prove that at least in the case of sufficiently large complete graphs and cycles, the situation is as expected.

\begin{theorem}\label{thm:ar_proper}
Let $F$ be a graph isomorphic to either a cycle on at least $7$ vertices or a complete graph on at least $19$ vertices. Then there exist constants $c,C > 0$ such that 
\begin{equation*}
     \lim_{n\to \infty}\Pr[\Gnp \aramcolp F] =
       \begin{cases}           
         1,&\text{if \(p \geq Cn^{-1/m_2(F)}\)}, \\
         0,&\text{if \(p \leq cn^{-1/m_2(F)}\)}.
       \end{cases}
  \end{equation*}
\end{theorem}
 
We remark that our bounds on the minimum size of the cliques resp.\ cycles are simply a consequence of our proof and probably not tight. As far as we know, the result actually could hold for all cliques and cycles of size at least $4$. 

 
\subsection{Outline of the Proof and organisation of the paper}\label{sec:outline}
 The main goal of this paper is to provide a unifying framework for proving $0$-statements for Ramsey-type properties. The main idea is to view the problem from an algorithmic perspective: we aim at providing an algorithm that finds the desired colouring with high probability. To do this we take the given random hypergraph $\Hknp$ as input and first 'strip of' easily colourable edges, where the definition of 'easily colourable' depends on the type of the given Ramsey problem. We then argue that whatever remains after the end of this stripping procedure can be partitioned into blocks that can be coloured separately. Our key result (Theorem~\ref{lemma:main}) states that with probability $1-o(1)$ these blocks will  have size at most some constant $L$ that depends (in some well-understood way) on the graph $F$. It is well known that in a typical random hypergraph with density $n^{-\alpha}$ all subgraphs of constant size have density at most $1/\alpha$. This implies that it suffices to prove that a statement of the form
\begin{equation}\label{eq:deterministic}
\text{all $\ell$-graphs } G\text{ with $m(G)\le m_\ell(F)$ satisfy}\quad G\nanycol F
\end{equation}
holds \emph{deterministically},  where by $\anycol$ we mean any of the discussed Ramsey properties. Note that any graph with density $m_\ell(F)$ appears in $\Hknp$ with constant probability for $p=cn^{-m_\ell(F)}$ (cf. proof of Corollary~\ref{lemma:main_density}  for details). Thus, the condition in \eqref{eq:deterministic} is actually {\em necessary} for the $0$-statement to hold. 
Formally, we call a graph $G$ an \emph{obstruction} for $F$ if $m(G)\le m_\ell(F)$ and $ G\anycol F$.
Note that such obstructing graphs $G$ indeed do exist. For some Ramsey type problems there are only a few, for others there exist infinitely many. We comment on that in more detail later. 
Our  aim is to show that the condition in \eqref{eq:deterministic} is also {\em sufficient}, i.e.\ in order to prove the $0$-statement it is sufficient to show that obstructions do not exist. 
We summarize this in the following ``meta-theorem''.

\begin{metathm}
Let $F$ be an $\ell$-graph for which \eqref{eq:deterministic} holds. Then
\begin{equation*}
  \lim_{n\to \infty}\Pr[\Hknp \anycol F] =
    \begin{cases}           
       1,&\text{if \(p \geq Cn^{-1/m_\ell(F)}\)}, \\
       0,&\text{if \(p \leq cn^{-1/m_\ell(F)}\)}.
    \end{cases}
\end{equation*}
\end{metathm}

Recall from the previous section that the $1$-statements are known to hold for all Ramsey problems considered in this paper. The key statement of our meta theorem is thus that the bound from the $1$-statement is actually tight, whenever  \eqref{eq:deterministic} holds.

In Section~\ref{sec:grow_approach} we prove our framework theorem, Theorem~\ref{lemma:main}. In Section~\ref{sec:applications} we provide the proofs for Theorems~\ref{thm:ramsey_cliques},~\ref{thm:2_bnd} and~\ref{thm:ar_proper} by showing deterministic statements corresponding to~\eqref{eq:deterministic}.


 \section{A general framework}\label{sec:grow_approach}
\subsection{Outline of the Method}

The key idea for the proof of  the Meta-Theorem from Section~\ref{sec:outline} is to introduce appropriate notions that capture the structure of overlapping copies of $F$. In the following definitions we always assume that $F$ contains at least two edges.

\begin{definition}[$F$-equivalence]
Given $\ell$-graphs $F$ and $G$, we say that two edges $e_1, e_2 \in E(G)$ are \emph{$F$-equivalent}, with  notation $e_1 \equiv_F e_2$, if for every $F$-copy $F'$ in $G$ we have $e_1 \in E(F')$ if and only if $e_2 \in E(F')$.
\end{definition}

\begin{definition}
Given an $\ell$-graph $F$ we define $\gamma(F)$ to be the largest intersection of two distinct edges in $F$, i.e.
$$ \gamma(F) := \max \{ |e_1 \cap e_2| \; : \; e_1, e_2 \in E(F) \text{ and } e_1 \neq e_2\}. $$
\end{definition}

\begin{definition}[$F$-closed property]
For given $\ell$-graphs $F$ and $G$, we define the property of being \emph{$F$-closed} as follows:
\begin{itemize}
\item an edge $e \in E(G)$ is \emph{$F$-closed} if

\begin{enumerate}[(a)]
\item $\gamma(F) = \ell - 1$ and $e$ belongs to at least two $F$-copies in $G$ or
\item $\gamma(F) < \ell - 1$ and $e$ belongs to at least two $F$-copies in $G$ and no edge $e' \in E(G) \setminus \{e\}$ is $F$-equivalent to $e$,
\end{enumerate}
\item an $F$-copy $F'$ in $G$ is \emph{$F$-closed} if at least three edges from $E(F')$ are closed,
\item the $\ell$-graph $G$ is \emph{$F$-closed} if every vertex and edge of $G$ belongs to at least one $F$-copy and every $F$-copy in $G$ is closed.
\end{itemize}
If the $\ell$-graph $F$ is clear from the context, we simply write \emph{closed}.
\end{definition}

\begin{definition}[$F$-blocks]
Given $\ell$-graphs $F$ and $G$ such that $G$ is $F$-closed, we say that $G$ is an $F$-block if for every non-empty proper subset of edges $E' \subsetneq E(G)$ there exists an $F$-copy $F'$ in $G$ such that $E(F') \cap E' \neq \emptyset$ and $E(F') \setminus E' \neq \emptyset$ (in other words, there exists an $F$-copy which partially lies in $E'$).
\end{definition}



With these definitions at hand we can now formulate our key result:

\begin{theorem} \label{lemma:main}
Let $\ell \ge 2$ be an integer and $F$ a strictly $\ell$-balanced $\ell$-graph such that either $F$ has exactly three edges and $\gamma(F) = \ell - 1$ or $F$ contains at least $4$ edges. Then there exist constants $c, L > 0$ such that for $p \leq cn^{-1/m_{\ell}(F)}$, $G \sim \Hknp$ satisfies  w.h.p. that every $F$-block $B \subseteq G$ contains at most $L$ vertices.
\end{theorem}

In all our applications we will use the following corollary of Theorem \ref{lemma:main} which gives a bound on the density $m$ of $F$-blocks. 

\begin{corollary} \label{lemma:main_density}
Let $\ell \ge 2$ be an integer and $F$ a strictly $\ell$-balanced $\ell$-graph such that either $\gamma(F) = \ell - 1$ and $F$ contains at least $3$ edges or $F$ contains at least $4$ edges. Then there exists a constant $c > 0$ such that for $p \le cn^{-1/m_\ell(F)}$, $G \sim \Hknp$ w.h.p. satisfies that for every $F$-block $B \subseteq G$ we have $m(B) \le m_\ell(F)$. Moreover, if $p \ll n^{-1/m_\ell(F)}$ then  strict inequality holds. 
\end{corollary}


We conclude this section with a basic property of $F$-closed graphs that will be used throughout the applications.

\begin{lemma} \label{lemma:block_split}
Let $F$ be an $\ell$-graph. Then if an $\ell$-graph $G$ is $F$-closed, there exists a partitioning $E(G) = E_1 \cup \ldots \cup E_k$, for some $k \in \mathbb{N}$, such that each subgraph $B_i$ induced by the set of edges $E_i$ is an $F$-block and each $F$-copy in $G$ is entirely contained in some block $B_i$.
\end{lemma}
\begin{proof}
Let $G$ be an $F$-closed $\ell$-graph and consider a smallest non-empty subset of edges $E' \subseteq E(G)$ such that every $F$-copy is either completely contained in $E'$ or avoids edges in $E'$. Observe that if an $F$-copy $F'$ in $G$ contains an edge $e \in E'$, then by the choice of $E'$ we have $E(F') \subseteq E'$. Similarly, if an $F$-copy $F'$ in $G$ contains an edge $e \in E(G) \setminus E'$ then $E(F') \subseteq E(G) \setminus E'$. Therefore, every edge $e \in E'$, resp. $e \in E(G) \setminus E'$ which was $F$-closed in $G$ remains $F$-closed in $G[E']$, resp. $G \setminus E'$, thus both $G[E']$ and $G \setminus E'$ are $F$-closed. By the minimality of $E'$ it follows that $G[E']$ is an $F$-block. We can now set $E_1 := E'$ and repeat the procedure on $G' := G \setminus E'$. In this way we obtain the desired partition $E_1, \ldots, E_k$.
\end{proof}


\subsection{Some useful facts} \label{sec:prelim}
The following lemma is a standard exercise in graph theory that we leave to the reader. 
\begin{lemma}[$k$-degeneracy] \label{thm:k-deg}
Let $G$ be a graph with $m(G) \le k$ for some $k \in \mathbb{R}$. Then there exists an ordering $(v_1, \ldots, v_n)$ of the vertices of $G$ such that
$$ |N(v_i) \cap \{v_1, \ldots, v_{i-1}\}| \le \lfloor 2k \rfloor $$
for every $i \in [n]$.
\end{lemma}

The proof of the following fact follows easily from Hall's theorem, cf.\ e.g.\cite{NenadovSteger:2014}.

\begin{lemma} \label{thm:k-orient}
Let $G$ be a graph with $m(G) \le k$ for some $k \in \mathbb{N}$. Then there exists an orientation of the edges of $G$ such that in the resulting directed graph each vertex has out-degree at most $k$.
\end{lemma}


\begin{lemma}[Markov's Inequality] \label{thm:markov}
Let $X$ be a non-negative random variable. For all $t > 0$ we have $\Pr[X \geq t] \le \frac{\mathbb{E}[X]}{t}$.
\end{lemma}

\subsection{Proof of Theorem \ref{lemma:main}} \label{sec:proof_of_main}
Here we show that $F$-blocks are with high probability  only of constant size (Theorem~\ref{lemma:main}).
Before we prove Theorem \ref{lemma:main}, we first show how it implies Corollary \ref{lemma:main_density}.

\begin{proof}[Proof of Corollary \ref{lemma:main_density}]
Let $L$ and $c$ be constants given by Theorem \ref{lemma:main} when applied to an $\ell$-graph $F$. Without loss of generality, we may assume that $c < 1$. We first consider the case $p \le cn^{-1/m_\ell(F)}$.

Let $\alpha \in \mathbb{R}$ be a strictly positive constant such that for every $\ell$-graph $S$ on at most $L$ vertices with $m(S) > m_\ell(F)$ we have $m(S) \ge m_\ell(F) + \alpha$. More formally, we define 
an $\alpha>0$ as follows,
$$ \alpha := \min \{  m(S) - m_\ell(F)  \mid v(S) \le L \; \text{ and } \; m(S) > m_\ell(F) \}.$$
Since there are only finitely many such $\ell$-graphs $S$, $\alpha$ is well-defined. Consider now some $\ell$-graph $S$ on at most $L$ vertices with $m(S) \ge m_\ell(F) + \alpha$ and let $S' \subseteq S$ be a subgraph such that $e(S') / v(S') = m(S)$. Let $X_{S'}$ be  the random variable which denotes the number of $S'$-copies in $G$. 
Then the expected number $\EE X_{S'}$ of $S'$-copies in $G \sim \Hknp$ is at most
\begin{align*}
\EE X_{S'} &\le n^{v(S')} p^{e(S')} \le n^{v(S') - e(S') / m_\ell(F)} \\
&= \left( n^{1 - m(S')/m_\ell(F)} \right)^{v(S')} \le n^{- \alpha \cdot v(S')/m_\ell(F)} = o(1).
\end{align*}
Therefore, by Markov's inequality (Lemma \ref{thm:markov}) we have
$$\Pr[G \; \text{contains an} \; S\text{-copy}] \le \Pr[G\; \text{contains an} \; S'\text{-copy}] = \Pr[X_{S'} \ge 1] \le \EE X_{S'}.$$
 As there exist less than $2^{\binom{L}{\ell}}$ different $\ell$-graphs on at most $L$ vertices, a union-bound over all such $\ell$-graphs thus also gives
$$ \Pr[ \exists S \subseteq G \; \text{ such that } \; v(S) \le L \; \text{and} \; m(S) > m_\ell(F)] = o(1).$$
In particular, since w.h.p.\ $G$ is such that every $F$-block $B \subseteq G$ contains at most $L$ vertices it follows that $m(B) \le m_\ell(F)$, as required.

Let us now assume that $p \ll n^{-1/m_\ell(F)}$. Similarly as in the previous case, if $S$ is an $\ell$-graph on at most $L$ vertices with $m(S) \ge m_\ell(F)$, then for $p \ll n^{-1/m_\ell(F)}$ we have
that the expected number of $S'$-copies is 
$$ \EE X_{S'} \le n^{v(S')} p^{e(S')} = o(n^{v(S') - e(S') / m_\ell(F)}) = o(1),$$ 
where $S' \subseteq S$ is such that $e(S') / v(S') = m(S)$. The same argument as before shows that $G$ contains no copy of $S$, which finishes the proof.
\end{proof}

\begin{proof}[Proof of Theorem \ref{lemma:main}]
Our proof is a generalization of the approach from~\cite{NenadovSteger:2014Ramsey} to hypergraphs and general Ramsey problems.
The proof is essentially a first moment argument. We enumerate all possible $F$-blocks on more than $L$ vertices and show that the probability that one or more of them appears in $G \sim \Hknp$ is $o(1)$. The difficulty lies in the fact that straightforward enumerations (like choosing subsets of edges) do not work: we have too many choices. We thus have to design a more efficient way to encode $F$-blocks. To do that we make use of Algorithm~\ref{alg:grow_seq} that enumerates $F$-copies of a block in some clever way.

{\LinesNumbered
\begin{algorithm}[h]\label{alg:grow_seq}
  \DontPrintSemicolon
  $F_0 \leftarrow$ an arbitrary $F$-copy in $B$\;
  $G_0 \leftarrow F_0$\;
  $i \leftarrow 0$\;
  \While{$G_i \neq B$}{
    $i \leftarrow i + 1$\;
    \eIf{$G_{i-1}$ contains an $F$-copy which is not closed}{
      $j \leftarrow$ smallest index $j < i$ such that $F_j$ is not closed\;
      $e \leftarrow$ 
      an edge in $F_j$ which is not closed in $G_{i-1}$ but closed in $B$ \label{line:choose_e}\;
      $F_i \leftarrow$ an $F$-copy in $B$ but not $G_{i-1}$ which contains $e$\label{line:open-case}\;
    }{
      $F_i \leftarrow$ an arbitrary $F$-copy in $B$ but not
      $G_{i-1}$ which intersects $G_{i-1}$ in at least one edge\label{line:closed-case}\;
    }
    $G_i \leftarrow G_{i-1} \cup F_{i}$\;
  }
  $s \leftarrow i$
  \caption{Construction of a grow sequence for an $F$-block $B$.}\label{algo:grow-sequence-algo}
\end{algorithm}}

Let $B$ be an $F$-block. 
Algorithm \ref{algo:grow-sequence-algo} maps $B$ to a sequence  $(F_0, \dotsc, F_s)$ of copies of $F$. In order to see that the algorithm is well-defined it suffices to show that lines \ref{line:open-case} and \ref{line:closed-case} can always be executed. For line \ref{line:open-case} this follows directly from the condition in the if-statement: an $F$-copy that is not yet closed contains an edge $e$ that is closed in $B$ but not yet in $G_{i-1}$. As in line \ref{line:choose_e} we choose exactly such an edge, the desired copy in line \ref{line:open-case} exists. Similarly, if at some point the execution of  line \ref{line:closed-case} would not be possible, this would imply that there exists a subgraph $G_i \subsetneq B$ such that every $F$-copy in $B$ is completely contained in either $G_i$ or $\overline{G}_i = B \setminus E(G_i)$. Since $G_i$ is non-empty (it contains $F_0$) this contradicts the assumption that $B$ is an $F$-block. Thus line \ref{line:closed-case} is well-defined. Finally, as the number of edges in $G_i$ increases with each iteration and $E(G_i) \subseteq E(B)$, at some point $G_i$ will be equal to $B$ and the algorithm will stop.

Note that the sequence $(F_0, \dotsc, F_s)$ fully describes a run of the algorithm. 
We call it a \emph{grow sequence} for $B$ and each $F_i$ in
it a \emph{step} of the sequence, $0 \leq i \leq s$. 
Given some grow sequence $S := (F_0, \dotsc, F_s)$ for $B$ we can easily reconstruct $B$ as the union of all $F_i$, $0 \leq i \leq s$.
We now turn to the question of how to enumerate such sequences efficiently.

Let us fix an arbitrary labeling of the vertices of $F$, say $V(F) = \{w_1, \ldots, w_{v(F)}\}$. Every $F$-copy in $B$ can be specified by an injective mapping $f: V(F) \rightarrow V(B)$, thus we can represent every $F$-copy in $B$ as a $v(F)$-tuple of vertices of $B$ where the $i$-th element of the tuple determines $f(w_i)$, for $1 \le i \le v(F)$. Accordingly, we could represent every grow sequence as a sequences of $v(F)$-tuples of vertices in $V(B)$. Unfortunately,  such an encoding is still too inefficient. We improve on this by using the fact that every $F$-copy $F_i$ from a grow sequence $(F_0, \ldots, F_s)$
has a non-empty intersection with $F_0 \cup \ldots \cup F_{i-1}$ .
We now make this more precise.
 %

We distinguish three step types. We call $F_0$ the \emph{first}
step.
For $i \geq 1$ we call the step $F_i$ \emph{regular} if the intersecting subgraph $G_{i-1} \cap F_i$ corresponds to exactly one edge, and \emph{degenerate} otherwise. 
In the first moment argument that we elaborate on below we choose the type of each step (regular or degenerate). For each type we then have to multiply the number of choices by the probability that the new edges (the edges in $E(F_i)\setminus E(G_{i-1})$) are present in $G$.

For a regular step $F_i$ created in line \ref{line:open-case}, the intersection with $G_{i-1}$ corresponds exactly to a non-closed edge $e$ in  $F_j$, where $j<i$ is the smallest index $j < i$ such that $F_j$ is not closed. Note that the index $j$ can be uniquely reconstructed from the graph $G_i$. That is, we do not have to choose it. This edge can be chosen in $e(F)$ ways. Furthermore, we have to choose which vertices in $F_i$ correspond to these vertices, giving another factor of $v(F)^\ell$. It remains to choose the other $v(F)-\ell$ new vertices of $F_i$, which in turn describe the $e(F)-1$ new edges that are required to be present. The total contribution of such a step is thus
\begin{equation}\label{eq:regular-open}
  e(F)v(F)^{\ell} n^{v(F)-\ell}p^{e(F)-1} \leq e(F)v(F)^{\ell} c^{e(F)-1} \leq c < 1,
\end{equation}
where $c$ is the constant in Theorem~\ref{lemma:main} which we
choose small enough for the above to hold.

In contrast to regular steps created in line \ref{line:open-case}, if a regular step $F_i$ is created in line \ref{line:closed-case} then a copy $F_j$ which contains an  intersecting edge of $F_i$ and $G_{i-1}$ is not fully determined by $G_{i-1}$ and we need to choose
it. By construction, the $\ell$-graph $G_{i-1}$ contains at most $v(F)\cdot i$
 vertices, thus there are at most $(v(F) \cdot i)^\ell$ choices for the vertices in the attachment edge in $G_{i-1}$ and the contribution of such a step is
\begin{equation}\label{eq:regular-closed}
  e(F) (v(F)\cdot i)^{\ell} n^{v(F)-\ell}p^{e(F)-1} \stackrel{\eqref{eq:regular-open}}{\le}  i^{\ell},
\end{equation}
again using the assumptions on the choice of $c$ in $(\ref{eq:regular-open})$. 

Now consider the case of degenerate steps, i.e.\ those for which \(H := F_i
\cap G_{i-1}\) satisfies $v(H) > \ell$. We can choose which vertices of $G_{i-1}$ correspond to $H$ in $(v(F) \cdot i)^{v(H)}$ many ways. Furthermore, recall that $F$ is strictly
$\ell$-balanced, so for any subgraph $H \subsetneq F$ with $v(H) > \ell$ we
have
\begin{equation*}
  \frac{e(H)-1}{v(H) - \ell} < \frac{e(F)- 1}{v(F)-\ell} = m_\ell(F)
\end{equation*}
and thus
\begin{equation}\label{eq:degenerate-hurts}
    \frac{e(F)-e(H)}{v(F)-v(H)} = \frac{(e(F)-1) - (e(H)-1)}{(v(F)-\ell) -
    (v(H)-\ell)} > m_\ell(F).
\end{equation}
This implies that we can choose a constant $\alpha > 0$ such that
for all $H \subsetneq F$ with $v(H) > \ell$ it holds that
\begin{equation*}
  v(F) - v(H) - \frac{e(F) - e(H)}{m_\ell(F)} < -\alpha.
\end{equation*}
Applying this to a degenerate step $F_i$, we obtain that the contribution is upper-bounded by 
\begin{align}\label{eq:degenerate}
   \sum_{\substack{H\subsetneq F\\v(H) > \ell}}(v(F)\cdot i)^{v(H)}
  n^{v(F)-v(H)}p^{e(F) - e(H)} \nonumber  
  & \leq
  i^{v(F)} \cdot c n^{-\alpha} \sum_{\substack{H\subsetneq F\\v(H) > \ell}}v(F)^{v(H)}   \nonumber
  \\
  & \leq i^{v(F)} \cdot c n^{-\alpha} \cdot v(F)^{v(F)} 2^{v(F)^2} \\
  & \le i^{v(F)} n^{-\alpha},\nonumber
\end{align}
where we again assume that $c$ is chosen small enough for the above to hold.

Thus, degenerate steps introduce a factor $i^{v(F)}n^{-\alpha}$, which suggests that sequences containing (constantly) many of them are very unlikely to appear in $G$. Similarly, regular steps created in line~\ref{line:open-case} introduce a factor of $c<1$, which suggests that sequences containing $\Theta(\log n)$ of these steps are also unlikely to appear in $G$.
The next claim provides bounds on the number of degenerate and regular steps created in line \ref{line:closed-case} that will allow us to conclude the proof.

\begin{claim}\label{claim:open-edges-bound}
  Let $S = (F_0, \ldots, F_s)$ be a grow sequence corresponding to an execution of  Algorithm~\ref{algo:grow-sequence-algo}. Then the following holds:
  \begin{enumerate}[(a)]
	 \item If $S$ contains at most $d$ degenerate steps, then $s \leq 3d \cdot v(F)$. \label{item:claim-open-edges-bound-1}
	 \item If a prefix $S'$ of $S$ contains at most $d$ degenerate steps, then every regular step $F_j$ in $S'$, with $j \ge 3d\cdot v(F) + 2$, is created in line \ref{line:open-case}. \label{item:claim-open-edges-bound-2}
  \end{enumerate}
\end{claim}


Intuitively, what Claim \ref{claim:open-edges-bound} tells us is that in a long grow sequence either there will be
many degenerate steps or most of the steps will be regular steps created in line \ref{line:open-case}.
Note that every degenerate step, as Equation~\eqref{eq:degenerate} shows, introduces a factor of $\Theta(n^{-\alpha+o(1)})$ to the expectation of the number of appearances of $S$ (for  $i=O(\log n)$) and regular step created in line \ref{line:open-case}  introduces a constant factor $c < 1$. We defer the formal proof of Claim \ref{claim:open-edges-bound} to the next section.



With the help of Claim \ref{claim:open-edges-bound} we can now finish our first moment argument. Set $\dmax := v(F)/\alpha + 1$
and $L := 3\dmax v(F) + 1$ and let $S = (F_0, \ldots, F_s)$ be a grow sequence of length more than $L$. By Claim~\ref{claim:open-edges-bound}\eqref{item:claim-open-edges-bound-1} every such sequence $S$
 must contain at least $\dmax$ degenerate steps. 
We now distinguish two cases. Let $s_d$ be the step in which the $\dmax$-th degenerate step occurs in $S$.
If $s_d <s_{\max}$, where $s_{\max} := v(F)\log n + \dmax+L$, then we set $S' := (F_0, \ldots, F_{s_d})$. Otherwise, we set $S' := (F_0, \ldots, F_{s_{\max}})$.
We prove that in both cases the expected number of possible grow sequences $S$ longer than $L$ which have a prefix $S'$ is $o(1)$. 

Observe that, in any case, $S'$ is a prefix of $S$ that contains at most $\dmax$ degenerate steps. Then, by Claim~\ref{claim:open-edges-bound}\eqref{item:claim-open-edges-bound-2}, if $F_i$ is a regular step from $S'$ created in line \ref{line:closed-case}, we have $i \le L$. Let us first consider the case when the $d_{\max}$-th degenerate step occurs before  step~$s_{\max}$, that is $s_d \in \{d_{\max}, \ldots, s_{\max} - 1\}$. For a fixed such $s_d$ there are $\binom{s_d-1}{d_{max} - 1}$ ways to choose steps in which the first $d_{\max} - 1$ degenerate steps have occured. We can now upper bound the expected number of such sequences $S'$ as follows
\begin{multline*}
  \sum_{s_d = d_{\max}}^{s_{\max}-1} \binoms{s_d - 1}{\dmax - 1}  n^{v(F)}
  \bigl( \underbrace{\strut  s_d^{v(F)}n^{-\alpha}}_{ \text{eq. }\eqref{eq:degenerate}}\bigr)^{\dmax}
  \bigl(\underbrace{(
  L)^{\ell}}_{\text{eq. } \eqref{eq:regular-closed}}\bigr)^{L} = \polylog(n) \cdot n^{v(F)}n^{-\alpha\cdot\dmax} = o(1).
\end{multline*}
Here we bound the contribution of the first step by $n^{v(F)}$, drop the
contribution of $c < 1$ for all regular steps created in line \ref{line:open-case}, and use the fact that only
the first $L + 1$ steps can be regular steps created in line \ref{line:closed-case}. 

Let us now consider the case $s_d \ge s_{\max}$. Note that then there are $d \in \{0, \ldots, \dmax\}$ degenerate steps within the first $s_{\max}$ steps. Similarly as in the previous case, we can upper bound the expected number of such sequences $S'$ as follows:
\begin{multline*}
  \sum_{d = 0}^{\dmax}\binoms{s_{\max}}{d} n^{v(F)}
  \bigl( \underbrace{\strut  s_{\max}^{v(F)}n^{-\alpha}}_{ \text{eq. }\eqref{eq:degenerate}}\bigr)^{d}
 \bigl(\underbrace{(
  L)^{\ell}}_{\text{eq. } \eqref{eq:regular-closed}}\bigr)^{L} 
  \underbrace{\strut c^{s_{\max}-d-L}}_{\text{eq. } \eqref{eq:regular-open}} \\= \polylog(n)\cdot n^{v(F)}\cdot c^{s_{\max}-d-L} =
  \polylog(n)\cdot
  2^{v(F) \log n }
  c^{v(F)\log n}
   = o(1),
\end{multline*}
where we used the fact that $c$ is small enough and in particular smaller than $1/2$.  

We can now conclude that the probability that $G$ contains a possible grow sequence  $S$ of length longer than $L$ as follows
$$ \Pr[S \text{ of length at least } L] \le \Pr[S \text{ contains a prefix $S'$ as described}] = o(1),$$
where the last inequality follows from Markov's inequality. Thus, with probability $1-o(1)$, every $F$-block in $G$ 
contains at most $v(F)\cdot (L+1)$ vertices.
\end{proof} 


\subsubsection{Proof of Claim~\ref{claim:open-edges-bound}}
\label{sec:proof-open-edges-lemma}

Let $S_i := (F_0, \ldots, F_i)$, for $0 \leq i \leq s$. For any $S_i$ and any
regular step $F_j$, $j \leq i$ we call the
edge $e := E(G_{j-1}) \cap E(F_j)$ the \emph{attachment edge} of $F_j$ and
the vertices in $V(F_{j}) \setminus V(G_{j-1})$ the \emph{inner vertices} of
$F_j$. For $j\le i$, we say that a regular step $F_j$ is \emph{fully-open} in $S_i$ if $\bigcup_{j' = j + 1}^i V(F_{j'})$ does not contain any inner vertex of $F_j$ (i.e., the inner vertices of $F_j$ have not been touched by any of the copies $F_{j+1},\ldots,F_i$). The first step $F_0$ is always fully-open by definition, and all its vertices are inner. Finally, we denote by 
$\reg(S_i)$, $\deg(S_i)$ and $\fo(S_i)$ the number of regular, degenerate and fully-open steps in $S_i$. 

It follows from the definition that a newly added regular step $F_i$ is fully-open in $S_i$. Next, we show a series of claims which will be used later in the proof of Claim~\ref{claim:open-edges-bound}. 

\begin{claim}\label{claim:tilde-F-is-F-f+e}
  Let $F$ be a strictly $\ell$-balanced $\ell$-graph with at least three edges. Furthermore, let $G$ be an arbitrary $\ell$-graph and $e\in E(G)$ an edge in $G$. Let $F_e$ be an $F$-copy such that  $G \cap F_e = (e,\{e\})$. 
%
  Then all $F$-copies $\tilde F$ in $G^+ := G \cup F_e$ which are
  not contained in $G$ have the form
  \begin{equation*}
    \tilde F = F_e - e + \tilde e := \bigl((V(F_e) \setminus
    e) \cup \tilde e, (E(F_e) \setminus \{e\}) \cup \{\tilde e\}\bigr),
  \end{equation*}
  where $\tilde e \in E(G)$ and $\abs{\tilde e \cap e} > \gamma(F)$,
  cf.~Figure~\ref{fig:tilde-F-is-F-f+e}.
\end{claim}
\begin{figure}[h]
  \centering
  \begin{tikzpicture}[scale=0.7]
    \draw[dotted] (-2, -0.4) .. controls (-1.3,1) and (0.3,0.9) .. (0.3, 0.9);
    \draw[dotted] (9, -0.4) .. controls (8.3,1) and (6.7,0.9) .. (6.7, 0.9);
    \fill[rounded corners=5pt,fill=white,draw=black] (0,0) rectangle (7,1);
    \draw (0,0.7) .. controls (0,5) and (7,5) .. (7,0.7);
    \draw[dashed] (0.2,0.8) .. controls (0.2,4.8) and (6.8,4.8) .. (6.8,0.8);
    \draw[dashed, rounded corners=5pt] (3, 0.8) -- (6.8, 0.8) -- (6.8, -1) -- (5, -1) -- (5, 0.2) -- (3,0.2) -- cycle;
    \draw[dashed] (0.2, 0.8) -- (3,0.8);
    \node (G) at (-1.5,-0.4) {$G$};
    \node (e) at (1.5,0.4) {$e$};
    \node (et) at (5.9,-0.5) {$\tilde e$};
    \node (Ft) at (3.5, 2.2) {$F_e - e + \tilde e$};
  \end{tikzpicture}
  \caption{The possible copies of $F$ created in a regular step. The solid
    lines represent $F_e$, the dashed ones $\tilde F$.}
  \label{fig:tilde-F-is-F-f+e}
\end{figure}
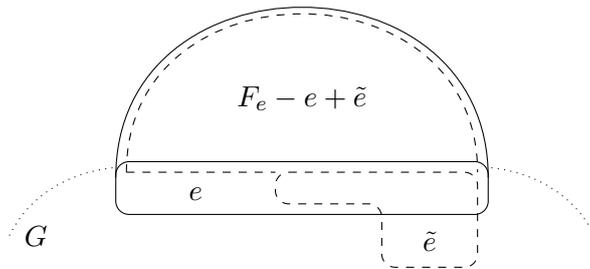
\begin{proof}
  \newcommand{\new}{\text{new}}
  \newcommand{\old}{\text{old}}
  Let $\tilde F$ be some $F$-copy in $G^+$ which is not fully
  contained in $G$. If $\tilde F = F_e$, then the lemma is true for $\tilde e =
  e$, so we assume $\tilde F \neq F_e$.

  Let $\tilde e$ be an arbitrary edge of $\tilde F$ which is not contained
  in $E(F_e)$. Note that this implies $\tilde e \in E(G)$.

  First we show that $E(\tilde F) \setminus \{\tilde e\}$ must be contained
  in $E(F_e)\setminus \{e\}$, which implies that the two sets are
  equal. Assume this is not true. Set $\tilde F_\new := \tilde F[V(F_e)]$,
  $\tilde F_\old := \tilde F[V(G)]$ and $\tilde F_{\new}^{+e} = \tilde F_\new + e$.  As we assumed that \(E(\tilde F)\setminus \{\tilde e\}
  \nsubseteq E(F_e)\setminus\{e\}\) we know that  \(\tilde F\) must contain an
  edge different from \(\tilde e\) that is not contained in 
  $E(F_e) \setminus \{e\}$,
  and is thus contained in \(E(G)\). This implies that
   $e(\tilde
  F_{\old}) \geq 2$. As \(\tilde F\) is not fully contained in \(G\) it must
  contain at least one edge of \(E(F_e) \setminus E(G)\), which in turn implies that \(e(\tilde
  F_\new^{+e}) \geq 2\). 
Subgraph $\tilde F_\old$ is a strict subgraph of $F$ as $\tilde F$ is not fully contained in $G$.
Moreover, $\tilde F_\new^{+e}$ is also a strict subgraph of $F$ as by definition
 $E(\tilde F_\new^{+e}) \subseteq E(F_e)$ and $|E(\tilde F_\new^{+e})| < |E(F_e)|$.

  One easily checks that regardless of whether \(e\) is  an edge of \(\tilde F_{\new}\) or not we have
  \begin{equation*}
    e(\tilde F) = e(\tilde F_\old) + e(\tilde F_\new^{+e}) - 1
    \quad\text{and}\quad
    v(\tilde F) \geq v(\tilde F_\old) + v(\tilde F_\new^{+e}) - \ell.
  \end{equation*}
  Thus
  \begin{equation*}
    m_\ell(\tilde F) = \frac{e(\tilde F)-1}{v(\tilde F)-\ell} \leq
    \frac{e(\tilde F_\old) - 1 + e(\tilde F_\new^{+e}) -1}{v(\tilde F_\old) -
      \ell + v(\tilde F_\new^{+e}) - \ell} < m_\ell(F),
  \end{equation*}
  which is a contradiction,  as $\tilde F$ is an $F$-copy. (Here the last inequality follows from the fact that $F$ is strictly $\ell$-balanced and $\tilde F_\new^{+e}, \tilde F_\old \subsetneq \tilde F$ are copies of a proper subgraph of $F$, each with at least $\ell+1$ vertices. ) Hence, our assumption $E(\tilde F) \setminus \{\tilde e\}\neq E(F_e)\setminus \{e\}$ is not valid.

  It remains to show that $\abs{\tilde e \cap e} > \gamma(F)$. 
  Let $X := \tilde{e} \setminus e$ and assume $\abs X \geq \ell - \gamma(F)$, i.e.
 $ \abs{\tilde e \cap e} \leq \gamma(F)$. As
  $\tilde F \setminus  \{ \tilde e\} = F_e \setminus \{e\}$ we know that
  no edge of $\tilde F$, except $\tilde e$, can contain a vertex in $X$.
  Let $H := \tilde F \setminus X$. By the previous observation we have
  $$
  v(H) = v(\tilde F) - \abs{X}\ge \ell+1 \quad \text{and} \quad e(H) = e(\tilde F) - 1\ge 2,
  $$
  thus 
  \begin{equation}\label{eq:claim9}
  m_\ell(H) \geq \frac{e(H) - 1}{v(H) - \ell} = \frac{e(\tilde F) - 1 - 1}{v(\tilde F) - \abs{X} - \ell} \geq 
  \frac{e(\tilde F) - 1 - 1}{v(\tilde F) - \ell - (\ell - \gamma(F))}, 
  \end{equation}
  where the last inequality holds because of the assumption on $X$.  We have $m_\ell(F) = \frac{e(\tilde F)-1}{v(\tilde F)-\ell}$ by the assumptions on $F$ being strictly $\ell$-balanced and  $m_\ell(F) \ge  \frac{1}{\ell - \gamma(F)}$ by the definition of $\gamma(F)$. Inequality $(\ref{eq:claim9})$ thus implies that $m_\ell(H) \ge m_\ell(F)$, which is a contradiction as $ H$ is a copy of a proper subgraph of $\tilde F$ with more than one edge. Thus we have $|\tilde e \cap e| > \gamma(F)$, as desired.
%
%
%
%
%
\end{proof}

Note that Claim~\ref{claim:tilde-F-is-F-f+e} implies that for $\ell$-graphs $F$ with $\gamma(F)=\ell-1$ (and in particular for {\em graphs}) we have that $G^+ = G \cup F_e$ does not contain {\em any} $F$-copy that intersects both $F_e\setminus e$ {\em and} $G\setminus e$. For these $\ell$-graphs the following claim is thus straightforward while for all other $\ell$-graphs it needs a small argument.

\begin{claim}\label{claim:regular-fully-open}
  Let $1\le j \le i$ and $F_j$ be a fully-open step in $S_i$. Let $e_j\in E(F_j)$ denote the attachment edge of $F_j$. Then any two distinct edges $e, e' \in E(F_j)\setminus\{e_j\}$ of $F_j$ are $F$-equivalent in $G_i$. 
\end{claim}

\begin{proof}
  As $F_j$ is fully-open in $S_i$ we know
that $G_i$ can be partitioned as $G_i = F_j \cup G_i'$ such that $F_j \cap G_i' = e_j$. From Claim~\ref{claim:tilde-F-is-F-f+e} we know that 
  any $F$-copy in $G_i$ which contains some edge $e\in E(F_j)\setminus\{e_j\}$ must also contain all other edges $e'\in E(F_j)\setminus\{e_j\}$, hence the claim follows.
  \end{proof}

For $i \geq 1$, let $\Delta(i)$ denote the number of fully-open copies
``destroyed'' by step $F_i$, i.e.\ let
\begin{equation*}
  \Delta(i) = \abs{\{j < i \mid \text{$F_j$ fully-open in $S_{i-1}$ but not in $S_i$}\}}.
\end{equation*}

\begin{claim} \label{claim:max_delta}
$$
\Delta(i) \le
\begin{cases}
1,&\mbox{if $F_i$ is a regular step}\\
v(F) - \ell +1,&\mbox{if $F_i$ is a degenerate step}. 
\end{cases}$$
\end{claim}
\begin{proof}

Fix any edge $e \in E(G_{i-1})$ and 
let $F_t$, $t<i$, be a step with $e\in E(F_t)$. Note that such a step has to exist as $e\in E(G_{i-1})$. Assume $e$ contains an inner vertex of some step $F_j$, $j<i$, which is fully-open in $S_{i-1}$.
If $t>j$ then $F_t$ contains an inner vertex 
of $F_j$, which contradicts our assumption that $F_j$ is fully-open in $S_{i-1}$. If $t<j$ then some inner vertex of $F_j$ is contained in an edge of $F_t$, which contradicts the definition of inner vertices of $F_j$. It follows that $t=j$ and $e\in E(F_j)$.


This easily implies the first part of the claim. Indeed, let $F_i$ be a regular step and $e_i = F_i \cap G_{i-1}$ its attachment edge. From the previous observation we have that $e_i$ can contain inner vertices of at most one $F$-copy $F_j$ which is fully-open in $S_{i-1}$, thus $\Delta(i) \le 1$ as required.

Next, similarly as in the case of edges we show that any vertex $v \in V(G_{i-1})$ can be an inner vertex of at most one $F$-copy $F_j$ which is fully-open in $S_{i-1}$. Fix any vertex $v\in V(G_{i-1})$ and assume that $F_j$
 is fully-open in $S_{i-1}$ with $v$ being its inner vertex. Let $F_t$, $t < i$, be a step containing $v$. 
 Then, by the same argument as above, it can not be that $t < j$. 
 By the definition of fully-open, the set $\cup_{j' = j+1}^{i-1} V(F_j')$ does not contain any inner vertex of $F_j$. In particular, if $t > j$ then this also holds for $F_t$. Therefore, $v$ can be an inner vertex only of  step $F_j$.


We can now derive the second part of the claim. Let $F_i$ be a degenerate step and $e \in E(F_i \cap G_{i-1})$ an arbitrary edge of $F_i$ which exists in $G_{i-1}$. By the first observation we have that $e$ contains inner vertices of at most one fully-open step in $S_{i-1}$. By the second observation, every vertex $v \in V(F_i \cap G_{i-1}) \setminus V(e)$ is an inner vertex of at most one fully-open step in $S_{i-1}$. In total, the step $F_i$ can touch inner vertices of at most $v(F) - \ell + 1$ fully-open copies.
\end{proof}

\begin{claim}\label{claim:consecutive-regular-bound}
  Let $F_i$ and $F_{i+1}$ be consecutive regular steps. If $\Delta(i) = 1$ then $\Delta(i+1) = 0$.
\end{claim}

\begin{proof}
  As $\Delta(i) = 1$ we know that $F_i$ is the first step which
  intersects the inner vertices of a fully-open step $F_j$ in $S_{i-1}$, for some $j <
  i$. Denote the attachment edges of $F_j$ and $F_i$ by $e_j$ and $e_i$, respectively. Before step $F_i$, by Claim \ref{claim:regular-fully-open} (if $\gamma(F) < \ell - 1$) and Claim \ref{claim:tilde-F-is-F-f+e} (if $\gamma(F) = \ell - 1$) the step $F_j$ had $e(F)-1 \ge 2$ edges which were not closed in $G_{i-1}$. We show below that step $F_i$ closes exactly one edge of $F_j$.   Thus, after the step $F_i$ the copy $F_j$ still contains at least one edge that is 
  not closed. Therefore,  in the $(i+1)$-iteration of the Algorithm~\ref{algo:grow-sequence-algo}, $F_{i+1}$ will be chosen in such a way that it intersects one of the edges of $F_j$ which are not yet closed. As $F_{i+1}$ is regular, it follows from the same arguments as in the proof of Claim \ref{claim:regular-fully-open} that it does not intersect the inner vertices of any other fully-open step in $S_i$ and we can conclude that $\Delta(i+1) = 0$.

 It remains to show that $F_i$ closes exactly one edge in $F_j$. We do this by a case distinction based on $\gamma(F)$.
 Assume first that $\gamma(F) = \ell - 1$ and consider some edge $e \in E(F_j) \setminus \{e_i,e_j\}$. By Claim~\ref{claim:tilde-F-is-F-f+e} the only $F$-copy in $G_{i-1}$ that contains $e$ is $F_j$. Moreover, again by Claim~\ref{claim:tilde-F-is-F-f+e} the only $F$-copy in $G_i$ which does not belong to $G_{i-1}$ is $F_i$. Since $e \notin E(F_i)$, $e$ also belongs to less than two copies in $G_i$ and thus it remains not closed.
  
Assume now that $\gamma(F) < \ell - 1$. Since in this case $F$ contains at least $4$ edges, let us consider any two distinct edges $e', e'' \in E(F_j) \setminus \{e_i, e_j\}$. First, it follows from Claim \ref{claim:regular-fully-open} that $e' \equiv_F e''$ in $G_{i-1}$. Furthermore, let us assume that there exists an $F$-copy $F'$ in $G_i$, not fully contained in $G_{i-1}$, which contains $e'$. Then, by Claim \ref{claim:tilde-F-is-F-f+e} there exists a unique such copy $F' = F_i - e_i + e'$ and $|e_i \cap e'| > \gamma(F)$. However, as $e_i$ and $e'$ both belong to the copy $F_j$, this contradicts the definition of $\gamma(F)$. Therefore, such an $F$-copy $F'$ does not exist and, by symmetry, the same is true for the edge $e''$. In other words, the property that an $F$-copy $\hat F$ in $G_{i}$ contains $e'$ if and only if it contains $e''$ remains true, thus $e'$ is not closed in $G_i$.
\end{proof}

As a final step before proving Claim \ref{claim:open-edges-bound},  we prove a lower bound on the number of fully-open steps that must be contained in any grow sequence of length $s$ with at most $d$ degenerate steps. Using Claim \ref{claim:consecutive-regular-bound}, the proof of the following claim is identical to the proof of Claim 11 from \cite{NenadovSteger:2014}.  We include it for the sake of completeness.

\begin{claim}\label{claim:fully-open-bound}
  For all $1 \leq i \leq s$ it holds that
  \begin{equation}\label{eq:fully-open-bound}
    \fo(S_i) \geq \reg(S_i)/2 -
    \deg(S_i)\cdot v(F).
  \end{equation}
\end{claim}

\begin{proof}
Let us denote by $\phi(i) := \reg(S_i)/2 - \deg(S_i) \cdot v(F)$ the right hand side of Equation \eqref{eq:fully-open-bound}. We use induction to prove the following slightly stronger statement,
$$
\fo(S_i) \geq 
\begin{cases}
\phi(i)&\text{ if $F_i$ is a regular step}\\
 \phi(i) + 1&\text{ if $F_i$ is a degenerate step,}
\end{cases}$$
for all $1 \le i \le s$.
One easily checks that this holds for $i=1$: if $F_1$ is a regular step then $\fo(S_1) = 1 > 1/2$, otherwise $\fo(S_1) = 0 > -v(F) + 1$. Consider now some $i \ge 2$.
If $F_i$ is a degenerate step then from Claim \ref{claim:max_delta} we have $\Delta(i) \le v(F) - \ell + 1 \le v(F) - 1$ and so $\fo(S_i) = \fo(S_{i - 1}) - \Delta(i) \geq \fo(S_{i - 1}) - v(F) + 1$. The claim now easily follows from $\reg(S_i) = \reg(S_{i - 1})$ and $\deg(S_i) = \deg(S_{i-1}) + 1$.

Otherwise, assume that $F_i$ is a regular step and let
$$j := \max\{ 1\le j < i \mid \Delta(j) > 0\text{ or } F_{j} \text{ is a degenerate step}\}.$$ 
Note that $j$ is well defined, as $\Delta(1)=1$. Further, by the definition of $j$, $F_{i'}$ is a regular step for all $j < i' \le i$, thus $\phi(i) = \phi(j) + (i -j)/2$. In addition, we deduce from $\Delta(i') = 0$ for $j<i'<i$ that all steps  $F_{i'}$ are fully-open in $S_{i-1}$. We thus have
$$
\fo(S_i) = \fo(S_j) + (1 - \Delta(i)) + (i - j - 1) = \fo(S_j) + i - j - \Delta(i).
$$
If $F_{j}$ is a degenerate step then the induction assumption implies $\fo(S_j) \geq \phi(j) + 1$. As $F_i$ is a regular step and thus $\Delta(i) \le 1$, this implies $\fo(S_i) \geq \phi(j) + i - j \ge \phi(i)$, as claimed. Finally, assume that $F_{j}$ is a regular copy. If $\Delta(i)=0$, then the claim follows trivially by the induction. Otherwise we have  $\Delta(i) = 1$ and as  $\Delta(j) = 1$ by Claim~\ref{claim:consecutive-regular-bound} we have that $i \geq j + 2$. Therefore 
$$\fo(S_i) = \fo(S_j) + i - j - 1 \geq \fo(S_j) + (i-j)/2 \ge\phi(j) + (i-j)/2=\phi(i),$$ similarly as before. This finishes the proof of the claim.
\end{proof}

Finally, we are ready to prove Claim~\ref{claim:open-edges-bound}.

\begin{proof}[Proof of Claim~\ref{claim:open-edges-bound}]
  We prove part (\ref{item:claim-open-edges-bound-1}) first.
  Let us assume that $S=(F_0,\ldots,F_s)$ contains at most $d$
  degenerate steps.  Every $F$-copy in $B:=\cup_{i=0}^s F_i$ is closed by the property of $S$, thus by Claim \ref{claim:regular-fully-open} there are no fully-open steps in $S$. By Claim~\ref{claim:fully-open-bound} this implies that
  \begin{equation}\label{eq:no-open-edges-bound}
    \deg(S)\cdot v(F)  \geq \reg(S) / 2
  \end{equation}
  must hold. We have $\deg(S) \leq d$ and $\reg(S) \geq s - d $ (the first step is neither degenerate nor regular). We obtain
  \begin{equation*}
    d\cdot v(F) \geq (s - d)/2.
  \end{equation*}
  Solving for $s$  we get
  \begin{equation*}
    s \leq 2 d \left(v(F) + 1/2 \right) \leq 3d\cdot v(F),
  \end{equation*}
  which proves the first part.

  For  part~(\ref{item:claim-open-edges-bound-2}) of Claim~\ref{claim:open-edges-bound} let $S_i$ be a  prefix of $S$, for some $1 \leq i \leq s$, with at most $d$ degenerate steps. Note that before any  regular step $F_j$, $j \leq i$ created in line \ref{line:closed-case}, all $F$-copies of $G_{j-1}$ are closed and thus by Claim \ref{claim:regular-fully-open} we have $\fo(S_{j-1}) = 0$. Similarly as above, we have $\deg(S_{j-1}) v(F) \geq \reg(S_{j-1})/2$. 
  As we know that $\deg(S_{j-1}) \leq d$ we obtain $ j - 1 \leq 3 d v(F)$, which concludes the proof.
\end{proof}

\section{Applications}\label{sec:applications}


\subsection{Anti-Ramsey property -- proper coloring} \label{sec:ar_proper}

The key ingredient for the proof of Theorem \ref{thm:ar_proper} is the following lemma whose proof we defer to the next section.

\begin{lemma} \label{lemma:ar_proper_cases}
Let $F$ be a graph isomorphic to either a cycle on at least $7$ vertices or a complete graph on at least $19$ vertices. Then for any graph $G$ such that $m(G) \le m_2(F)$ it holds that $G \naramcolp F$.
\end{lemma}


\begin{proof}[Proof of Theorem \ref{thm:ar_proper}]
Let $F$ be some graph as stated in the theorem and $c$ a constant given by Corollary~\ref{lemma:main_density} when applied to $F$. Let $p \le cn^{-1/m_2(F)}$ and $G \sim \Gnp$.  We use Algorithm \ref{algo:ar_proper} to find a proper coloring of $G$ without a rainbow $F$-copy.
{\LinesNumbered
\begin{algorithm}
  \DontPrintSemicolon
  $\hat G \leftarrow G$\;
  $\textrm{col} \leftarrow 0$\;
  \While{$\exists e_1, e_2 \in E(\hat G) \; : \; e_1 \equiv_F e_2$ in $\hat G$ and $e_1 \cap e_2 = \emptyset$}{
  	color $e_1, e_2$ with $\textrm{col}$\;
  	$\hat G \leftarrow \hat G \setminus \{e_1, e_2\}$ and $\textrm{col} \leftarrow \textrm{col} + 1$ \label{line:proper_f_equiv}\;
  }
  \While{$\exists e \in E(\hat G) \; : \; e$ does not belong to an $F$-copy}{
  	color $e$ with $\textrm{col}$\;
  	$\hat G \leftarrow \hat G \setminus \{e\}$ and $\textrm{col} \leftarrow \textrm{col} + 1$
  }
  Remove isolated vertices in $\hat G$\; 
  $\{B_1, \ldots, B_k\} \leftarrow $ $F$-blocks obtained by applying Lemma \ref{lemma:block_split} on $\hat G$    \;
  Color (properly) each $B_j$ without a rainbow $F$-copy using distinct sets of colors (cf.~text why this is possible) \label{line:proper_G_hat}  \vspace{0.2cm}
  \caption{Proper colouring without rainbow $F$-copy.} \label{algo:ar_proper}
  \end{algorithm}}

 To see the correctness of the algorithm, observe first that it suffices to argue that the graph $\hat G$  obtained in  line \ref{line:proper_after_removing} can be properly colored without a rainbow copy of $F$. Indeed, we only remove edges that are not contained in an $F$-copy (and can thus be colored arbitrarily) or pairs of (non-adjacent) edges that are both contained in exactly the same $F$-copies (and can thus not be contained in a rainbow copy, if we give them the same color). 

It thus remains to prove that line \ref{line:proper_G_hat} is indeed possible. We first show that the graph $\hat G$  is $F$-closed. 
Assume otherwise. Then there has to exist  an $F$-copy $F'$ which has at most two closed edges (as there are no vertices and edges which are not a part of an $F$-copy). If $F' \cong K_\ell$ then as $\ell \ge 19 $ there at least 
$\binom{\ell}{2} -2 > \ell$ edges of $E(F')$ which are not closed. One easily checks that this implies that there are two  edges $e_1, e_2 \in E(F')$ that satisfy $e_1 \cap e_2  = \emptyset$ and are not closed. Thus, $F'$ is the only $F$-copy to which $e_1$ and $e_2$ belong, implying that $e_1 \equiv_F e_2$. However, this can't be, as such a pair would have been removed in line~\ref{line:proper_f_equiv} of the algorithm.
If $F' \cong C_\ell$ then there are at least $\ell - 2 \geq 5 $ edges of $F'$ which are not closed and as $F'$ is a cycle two of those must be non-intersecting, again yielding a contradiction similarly as in the previous case.

So we know that $\hat G$ is $F$-closed. We thus can apply Lemma \ref{lemma:main} to deduce that we have  w.h.p. that each $F$-block $B$ in $G$ satisfies $m(B) \le m_2(F)$. By Lemma \ref{lemma:block_split}, coloring one block $B_i$ does not influence the coloring of any $F$-copy which does not lie in $B_i$ and all $B_i$'s are edge-disjoint. Finally, by Lemma \ref{lemma:ar_proper_cases} there exists a desired proper coloring of every block $B_i$, which gives a proper coloring of $\hat G$ (and  of the graph $G$) without a rainbow $F$-copy.
\end{proof}

\subsubsection{Proof of Lemma \ref{lemma:ar_proper_cases}}

We start with a technical observation that will help us prove the case of forbidden complete graphs.

\begin{claim} \label{claim:cliques_overlap}
Let $\ell \ge 4$ be an integer and let $G$ be a graph with $m(G) \le (\ell + 1)/2$. Then for any vertex $v \in V(G)$ and a subset $A \subseteq N_G(v)$ of size $|A| \le \ell + 1$, there exist at most $\lfloor 6 \cdot \ell / (\ell - 3)\rfloor$ vertices $w \in V(G) \setminus (A \cup \{v\})$ with the property that $G[A' \cup \{v, w\}] \cong K_{\ell}$ for some $A' \subseteq A$.
\end{claim}
\begin{proof}
First, note that if $G[A]$ does not contain a copy of $K_{\ell - 2}$ then there is no such vertex $w \in V(G) \setminus (A \cup \{v\})$. Therefore, we can assume that $|A| \ge \ell - 2$ and $G[A]$ contains at least $\binom{\ell - 2}{2}$ edges. Note that then $e(G[A \cup \{v\}]) \ge \binom{\ell - 2}{2} + \ell - 2$. Assume now that there are $k$ vertices $W = \{w_1, \ldots, w_k\} \subseteq V(G) \setminus (A \cup \{v\})$ with the described property. Then each such vertex $w_i$ has at least $\ell - 1$ neighbours among vertices in $A \cup \{v\}$, thus
\begin{align}
e(G[A \cup \{v\} \cup W]) &\ge \frac{(\ell-2)(\ell - 3)}{2} + \ell - 2 + k\cdot(\ell - 1) \notag \\
&= \frac{(\ell - 2)(\ell - 1)}{2} + k(\ell - 1) = (\ell - 1)(\ell/2 - 1 + k). \label{eq:contra1}
\end{align}
On the other hand, from $m(G) \le (\ell + 1)/2$ and $|A| \le \ell + 1$ we have
\begin{equation}
 e(G[A \cup \{v\} \cup W]) \le \frac{\ell + 1}{2} (\ell + 2 + k) = (\ell + 1)(\ell/2 + 1 + k/2). \label{eq:contra2}
\end{equation}
Finally, combining \eqref{eq:contra1} and \eqref{eq:contra2} gives $k \le 6 \cdot \ell / (\ell - 3)$ which concludes the proof of the claim as $k$ has to be an integer.
\end{proof}

\begin{proof}[Proof of Lemma \ref{lemma:ar_proper_cases} - complete graphs]
Let $\ell \ge 19$ and $G$ be a graph on $n$ vertices with $m(G) \le m_2(K_\ell) = (\ell + 1)/2$. By Lemma~\ref{thm:k-deg} there exists an ordering $v_1, \ldots, v_n$ of the vertices of $G$ such that \begin{equation} \label{eq:order}
|N(v_i) \cap \{v_1, \ldots, v_{i-1}\}| \le \ell + 1
\end{equation}
for every $i \in [n]$ and let $G_i := G[\{v_1, \ldots, v_i\}]$. Given a (partial) edge-coloring $p$ of $G$, we say that an edge $e \in E(G)$ is \emph{$i$-new} if no edge in $G_{i-1}$ is colored with $p(e)$. We will inductively find a proper coloring $p_i$ of $G_i$ such that the following holds,
\begin{enumerate}[(i)]
\item $G_i$ does not contain a rainbow copy of $K_\ell$ under coloring $p_i$,
\item for every $j \in [i]$: all but at most three edges incident to $v_j$ in $G_j$ are $j$-new, and
\item for every $j < r \le i$: if an edge $\{v_j, v_r\} \in E(G)$ is not $j$-new, then there exists a subset of vertices $S \subseteq \{v_1, \ldots, v_{j-1}\}$ such that $G[\{v_j, v_r\} \cup S] \cong K_\ell$.
\end{enumerate}
The base of the induction trivially holds, thus assume that the induction hypothesis holds for all $i < k$, for some $2 \le k \le n$.

Let $p_{k-1}$ be any coloring of $G_{k-1}$ which satisfies $(i)$-$(iii)$.  We create a coloring $p_k$ by extending the coloring $p_{k-1}$ to the edges incident to $v_k$ in $G_k$. Note that this implies that the only $K_\ell$-copies we have to take care of are those which contain the vertex $v_k$. Similarly, the only edges which might violate properties $(ii)$ and $(iii)$ are those incident to $v_k$. 

Let $v_{i_1}, \ldots, v_{i_q}$ be the neighbours of $v_k$ in $G_k$, with $i_{j} < i_{j+1}$ for all $j \in [q-1]$. It follows from \eqref{eq:order} that $q \le \ell + 1$. Initially, assign an arbitrary new color to each edge $\{v_k, v_{i_j}\}$ for $j \le \min\{q, \ell - 2\}$. Note that this leaves at most three edges of $G_k$ uncolored, thus the property $(ii)$ is guaranteed to be satisfied. If $q < \ell - 1$, then the vertex $v_k$ does not belong to any copy of $K_\ell$ in $G_k$ and properties $(i)$ and $(iii)$ remain satisfied as well -- in which case we are done. Therefore, from now on we assume that $q \in \{\ell - 1, \ell, \ell + 1\}$.

 Let $R = \{v_{i_{\ell-1}}, \ldots, v_{i_q}\}$ be the set of the remaining neighbours of $v_k$, i.e. endpoints of edges that are not yet colored. We first "clean" $R$ as follows: for any $v_j \in R$ for which there does not exist a subset $S \subseteq \{v_1, \ldots, v_{j-1}\}$ such that $G[S \cup \{v_j, v_k\}] \cong K_{\ell}$, assign an arbitrary new color to $\{v_j, v_k\}$ and set $R := R \setminus \{v_j\}$. Note that if $R = \emptyset$ after this procedure, then $v_k$ does not belong to a copy of $K_\ell$ in $G_k$ and it is easy to see that properties $(i)$-$(iii)$ are satisfied. Therefore, we can assume that $R \neq \emptyset$ and observe that any coloring we assign to the remaining edges will satisfy $(iii)$. Furthermore, note that every copy of $K_\ell$ which contains $v_k$ in $G_k$ also contains at least one vertex from $R$.

Before we proceed with the coloring of the remaining edges, we first make an observation about the coloring of the edges in $G_{k-1}$. Let $v_j \in R$ be an arbitrary vertex. An application of Claim~\ref{claim:cliques_overlap} to $A:=N(v_j)\cap\{v_1,\ldots,v_{j-1}\}$, which is by~\eqref{eq:order} at most $\ell+1$, yields that there exist at most
\begin{equation}
\lfloor6\ell/(\ell-3)\rfloor \stackrel{(\ell \ge 19)}{\le} 7 \label{eq:7}
\end{equation}
vertices $v_z \in V(G) \setminus (A \cup \{v_j\})$  such that there exists $S_z \subseteq A$ with $G[\{v_j, v_z\} \cup S_z] \cong K_\ell$. Since, by the definition of $R$, $v_k$ is such a vertex, it follows from \eqref{eq:7} and the proeprty $(iii)$ that there are at most $6$ vertices $v_z$, $j < z < k$ such that the edge $\{v_j, v_z\}$ is not $j$-new. Combining this observation with  property $(ii)$, we have that there are at most $9$ colors assigned to edges incident to $v_j$ which are also assigned to some edge in $G_{j-1}$. Let us denote the set of such colors with $C_j$ and
\begin{equation}
|C_j| \le 9. \label{eq:bound_C}
\end{equation}
With this observation at hand, we go back to the coloring of the remaining edges.

Let $W := \{v_{i_1}, \ldots, v_{i_{\ell - 2}}\}$.
Our aim now is as follows: for each vertex $v_j \in R$ we want to find pairwise disjoint $2$-sets $S_j \subseteq W$ such that either $S_j \notin E(G)$ or $p_{k-1}(S_j) \notin C_j$ and $p_{k-1}(S_j) \neq p_{k-1}(S_{j'})$ for distinct $v_j, v_{j'} \in R$. Then the coloring can be completed by setting $p_k(\{v_k, v_j\}) := p_{k-1}(S_j)$ if $S_j \in E(G)$ and assigning an arbitrary new color otherwise. Clearly, a rainbow $K_\ell$-copy which contains $v_k$ and $v_j \in R$ cannot contain both vertices in $S_j$, thus if it contains $v_k$ then it has to miss at least $|R|$ vertices from $W \cup R$. As $|W| \le \ell - 2$  this shows that no such rainbow $K_\ell$-copy exists, which finishes the proof.

We find these sets $S_j$ in a greedy fashion as follows. Let $R' := R$ and $W' := W$ and repeat the following until $R' = \emptyset$: if there exist two vertices $a,b\in W'$ such that $a$ and $b$ do not form an edge, choose $v_j \in R'$ arbitrarily and set $S_j :=\{a,b\}$, $R' := R' \setminus \{v_j\}$ and $W' := W' \setminus \{a, b\}$. Otherwise, choose $v_j \in R'$ arbitrarily and let $a, b \in W'$ be such that $p_{k-1}(\{a,b\}) \notin C_j$ and $p_{k-1}(S_j) \neq p_{k-1}(S_{j'})$ for previously defined sets $S_{j'}$. If this procedure exhausts $R'$, then by the construction of the sets $S_j$ we are done. Furthermore, since in each iteration the size of $R'$ decreases, it suffices to show that both cases are well-defined.

If there exists two vertices $a, b \in W'$  that do not form an edge then there is nothing to show. Therefore, we can assume that $W'$ induces a clique. Note that, for each $v_j \in R$, at most $11$ colors are forbidden; at most two because of the previously defined sets $S_{j'}$ and at most $9$ because of $C_j$. Thus, in order to show that we can find an edge $S_j$ in $W'$ which satisfies the desired property, it suffices to show that there are more than $11$ different colors appearing in the clique $W'$. Since $|R| \le 3$ and $|W| = \ell - 2$ we have $|W'| \ge \ell - 2 - 2 \cdot 2 = \ell - 6$ as long as $R' \neq \emptyset$. On the other hand, every proper coloring of a clique on at least $\ell - 6$ vertices contains at least $\ell - 7 > 11$ different colors, which finishes the proof.
\end{proof}

We remark that more careful counting of the number of different colors in the clique $W'$ gives a slightly better lower bound on $\ell$. Next, we prove the case of cycles.

\begin{proof}[Proof of Lemma \ref{lemma:ar_proper_cases} - cycles]
Let $\ell \ge 7$ and $G$ be a graph on $n$ vertices such that $m(G) \le m_2(C_\ell) = 1 + 1/(\ell -  2)$. Let us assume towards a contradiction that $G$ is a minimal graph with respect to the number of vertices such that $G \aramcolp C_\ell$.

First, observe that in $G$ no two vertices of degree $2$ are adjacent. To see this, let us assume that two such vertices $v_1, v_2 \in V(G)$ exist. Then $N(v_1) \cap N(v_2) = \emptyset$ as otherwise $v_1$ and $v_2$ do not belong to a $C_\ell$-copy thus contradicting the minimality of $G$. Therefore, the edges $e_1$ and $e_2$ incident to $v_1$ and $v_2$, different from the edge $\{v_1, v_2\}$, satisfy $e_1 \cap e_2 = \emptyset$. Furthermore, it follows again from the minimality of $G$ that
$$ G \setminus \{v_1, v_2\} \naramcolp C_\ell.$$
Consider an arbitrary coloring of $G \setminus \{v_1, v_2\}$ without a rainbow $C_\ell$-copy. We assign the same (new) color to $e_1$ and $e_2$. Observe that no rainbow $C_\ell$-copy can contain both $v_1$ and $v_2$. On the other hand, since $e_1 \equiv_{C_\ell} e_2$ in $G$ and there is no rainbow $C_\ell$-copy in $G \setminus \{v_1, v_2\}$ this implies $G \naramcolp C_\ell$, a contradiction.

Next, it is easy to see that $G$ does not contain a vertex $v$ of degree $1$ as such a vertex does not belong to a $C_\ell$-copy and would contradict the minimality of $G$. Therefore, $\delta(G) \ge 2$ and by the previous observation the set $V_2 \subseteq V(G)$ of all the vertices of degree $2$ is an independent set. We estimate the size of $V_2$ as follows,
$$ 2m(G) n \ge 2e(G)=\sum_{v \in G} \deg(v) \ge \sum_{v \in V_2}2 + \sum_{v \in V(G) \setminus V_2}3 = |V_2| \cdot 2 + (n - |V_2|) \cdot 3 $$
and therefore $|V_2| \ge (1 - 2/(\ell - 2)) n$. Since $V_2$ is an independent set this implies 
$$(1+\tfrac{1}{\ell-2})n \ge e(G) \ge e(V_2, V(G) \setminus V_2) = |V_2| \cdot 2 \ge (2 - \tfrac4{\ell - 2}) n.$$
One easily checks that this a contradiction for all $\ell\ge 8$. For $\ell=7$ we have that the left hand side is equal to the right hand side, which implies that the graph $G$ is bipartite.
Since $C_7$ is not bipartite, $G$ does not contain $C_\ell$-copy, implying the desired contradiction also in this case. 
\end{proof}

%
%

\subsection{Anti-Ramsey property -- 2-bounded colorings} \label{sec:ar_bounded}

Here we give a proof of Theorem \ref{thm:2_bnd}. We use the following three lemmas which provide a density condition of graphs that are not anti-Ramsey corresponding to the three cases from Theorem \ref{thm:2_bnd}.  We defer the proofs to the next subsection.

\begin{lemma} \label{lemma:nbounded_general}
Let $F$ be a strictly $2$-balanced graph on at least $4$ vertices which contains a cycle and is not isomorphic to $C_4$. Then for any graph $G$ such that $m(G) \le m_2(F)$ it holds that $G \naramcolk[2] F$.
\end{lemma}

\begin{lemma} \label{lemma:nbounded_C4}
For any graph $G$ such that $m(G) < m_2(C_4)$ it holds that $G \naramcolk[2] C_4$. Moreover, there exists a graph $G$ with $m(G) = m_2(C_4)$ such that $G \aramcolk[2] C_4$.
\end{lemma}

\begin{lemma} \label{lemma:nbounded_hyper}
Let $r, \ell \in \mathbb{N}$ be such that $2 \le \ell \le r - 1$ and $(\ell,r) \notin \{(2,3), (3,4)\}$. Then for any $\ell$-graph $G$ with $m(G) \le m_{\ell}(K_r^{(\ell)})$ it holds that $G \naramcolk[2] K_r^{(\ell)}$.
\end{lemma}

\begin{proof}[Proof of Theorem \ref{thm:2_bnd}]
Let $\ell \ge 2$ be an integer and consider some strictly $\ell$-balanced $\ell$-graph $F$ which satisfies one of the conditions of the theorem and let $c$ be a constant given by Corollary \ref{lemma:main_density} when applied to $F$. Let $G \sim \Hknp$ for $p$ which we will specify later. We use  Algorithm~\ref{algo:ar_2bnd} to find a $2$-bounded coloring of $G$ without a rainbow $F$-copy.

{\LinesNumbered
\begin{algorithm}
  \DontPrintSemicolon
  $\hat G \leftarrow G$\;
  $\textrm{col} \leftarrow 0$\;
  \While{$\exists e_1, e_2 \in E(\hat G) \; : \; e_1 \equiv_F e_2$ in $\hat G$}{
  	color $e_1, e_2$ with $\textrm{col}$\;
  	$\hat G \leftarrow \hat G \setminus \{e_1, e_2\}$ and $\textrm{col} \leftarrow \textrm{col} + 1$
  }
  \While{$\exists e \in E(\hat G) \; : \; e$ does not belong to an $F$-copy}{
  	color $e$ with $\textrm{col}$\;
  	$\hat G \leftarrow \hat G \setminus \{e\}$ and $\textrm{col} \leftarrow \textrm{col} + 1$
  }
 Remove isolated vertices in $\hat G$\; \label{line:proper_after_removing}
  $\{B_1, \ldots, B_k\} \leftarrow $ $F$-blocks obtained by applying Lemma \ref{lemma:block_split} with $\hat G$\;
  Color ($2$-bounded) each $B_i$ without a rainbow $F$-copy using a distinct set of colors
  \vspace{0.2cm}
  \caption{$2$-bounded colouring of $G$ without rainbow $F$-copy.} \label{algo:ar_2bnd}
\end{algorithm}}

The only difference between Algorithm \ref{algo:ar_proper} and Algorithm \ref{algo:ar_2bnd} is in the condition in line 3. In particular, in Algorithm \ref{algo:ar_2bnd} we don't require edges $e_1$ and $e_2$ to be disjoint. Following the same lines as in the proof of Theorem \ref{thm:ar_proper} together with Lemma \ref{lemma:nbounded_general} (provided $p \le cn^{-1/m_2(F)}$), Lemma \ref{lemma:nbounded_C4} (provided $F\cong C_4$ and $p \ll n^{-1/m_2(F)}$) and Lemma \ref{lemma:nbounded_hyper} (provided $p \le cn^{-1/m_\ell(F)}$) shows that w.h.p. $G$ is such that the Algorithm \ref{algo:ar_2bnd} finds the desired colouring.
\end{proof}

\subsubsection{Proof of Lemmas \ref{lemma:nbounded_general} and \ref{lemma:nbounded_C4}}

Proof of Lemma \ref{lemma:nbounded_general} splits into a couple of cases. We first state claims which cover these cases. Throughout this section, we say that a $2$-bounded coloring of edges incident to some vertex $v$ is \emph{maximal} if all but at most one color appears exactly twice.

\begin{claim} \label{claim:min_degree_col}
Let $G$ and $F$ be graphs such that $m(G) < \delta(F) - 1/2$. Then $G \naramcolk[2] F$.
\end{claim}
\begin{proof}
Consider some graph $F$ and assume towards a contradiction that there exists a graph $G$ on $n$ vertices with $m(G) < \delta(F) - 1/2$ such that $G \aramcolk[2] F$. Furthermore, let us assume that $G$ is a minimal such graph with respect to the number of vertices. It then follows from
$$\sum_{v \in V(G)} \deg(v) = 2e(G) \le 2m(G) n$$
that there exists a vertex $u \in V(G)$ with $\deg(u) \le 2m(G) < 2\delta(F) - 1$. Since $\deg(u) \in \mathbb{Z}$ we can further improve this bound to $\deg(u) \le \lfloor 2m(G) \rfloor \leq 2(\delta(F) - 1)$. Now consider an arbitrary maximal $2$-bounded coloring of the edges incident to $u$ and color $G - \{u\}$ using the minimality assumption. Then in any rainbow subgraph of $G$ the vertex $u$ has degree at most $\delta(F) - 1$, thus $u$ cannot belong to a rainbow $F$-copy. However, as there are no rainbow $F$-copies in $G - \{u\}$ we have a $2$-bounded coloring of $G$ without a rainbow $F$-copy, which is a contradiction. 
\end{proof}

The proof of the next claim uses similar ideas as the proof of Lemma \ref{lemma:ar_proper_cases} in the case of cycles.

\begin{claim} \label{claim:2_deg_adj}
Let $G$ and $F$ be graphs such that $m(G) < \delta(F) - 2/7$, $\delta(F) \ge 2$ and $F$ does not contain two adjacent vertices of degree $\delta(F)$. Then $G \naramcolk[2] F$.
\end{claim}
\begin{proof}
Let us consider some graph $F$ as in the statement of the claim and assume towards a contradiction that there exists a graph $G$ on $n$ vertices with $m(G) < \delta(F) - 2/7$ such that $G \aramcolk[2] F$. Furthermore, assume that $G$ is a minimal such graph with respect to the number of vertices.

First, we can assume that $\delta(G) \ge 2\delta(F) - 1$ as otherwise the claim follows from the same arguments as in the proof of Claim \ref{claim:min_degree_col}. Furthermore, similarly as in the proof of the cycle case of Lemma \ref{lemma:ar_proper_cases} we can show that $G$ does not contain two adjacent vertices $v_1, v_2 \in V(G)$ with $\deg(v_1) = \deg(v_2) = 2\delta(F) - 1$. Indeed, assume that two such vertices $v_1, v_2 \in V(G)$ exist. Then we color $G\setminus\{v_1,v_2\}$ by the minimality assumption without a rainbow $F$-copy, assign a new color to the edge $\{v_1, v_2\}$ and color the remaining edges incident to $v_1$ and $v_2$ both by a maximal $2$-bounded coloring. Then the degree of $v_1$ and $v_2$ in any rainbow subgraph $R \subseteq G$ is at most $\delta(F)$. If $\{v_1, v_2\} \subseteq R$ then $R \ncong F$ since $F$ does not contain two adjacent vertices of degree $\delta(F)$. Otherwise,  $v_1$ and $v_2$ can have degree at most $\delta(F)-1$ in $R$, which again implies that $R \ncong F$ or $v_1,v_2\not\in R$. Therefore, any rainbow $F$-copy has to lie completely in $G - \{v_1, v_2\}$ which is not possible. 

To summarize, we have $\delta(G) \ge 2\delta(F) - 1$ and the set $S \subseteq V(G)$ of all the vertices of degree exactly $2\delta(F) - 1$ is an independent set. We estimate the number of edges in $G$ as follows,
\begin{align*}
2m(G) n &\ge \sum_{v \in V(G)} \deg(v) \ge |S|(2\delta(F) - 1) + (n - |S|)2\delta(F) = n \cdot 2\delta(F) - |S|
\end{align*}
and thus $|S| \ge 2n (\delta(F) - m(G))$. Now $m(G) < \delta(F) - 2/7$ implies that $|S| > 4/7 \cdot n$. Since $S$ is an independent set, we further have
\begin{multline*}
(\delta(F) - 2/7)n \ge m(G)n\ge e(G) \ge e(S, V(G) \setminus S) \geq \\ |S| \cdot (2\delta(F) - 1) > n(8/7 \cdot \delta(F) - 4/7),
\end{multline*}
which easily implies $\delta(F) < 2$, hence a contradiction. Therefore, such graph a $G$ does not exist.
\end{proof}

\begin{claim} \label{claim:orientation}
Let $F$ and $G$ be graphs such that 
\begin{enumerate}[(i)]
\item $\lceil m(G) / 2 \rceil < m(F)$ or
\item $\lceil m(G) / 2 \rceil = m(F)$, $m(G) < \lceil m(G) \rceil$ and $\lceil m(G) \rceil$ is odd.
\end{enumerate}
Then $G \naramcolk[2] F$.
\end{claim}
\begin{proof}
Let $F$ and $G$ be graphs which satisfy condition $(i)$ of the claim. By Lemma~\ref{thm:k-orient} there exists an orientation of the edges of $G$ such that each vertex has out-degree at most $\lceil m(G) \rceil$. Let us consider one such orientation and arbitrarily pair the out-edges incident to each vertex. Assigning the same color to edges in each pair, in any rainbow (oriented) subgraph $R \subseteq G$ we have for the out-degree of any vertex $v\in V(R)$
\begin{equation} \label{eq:orient}
\deg^+_R(v) \le \left\lceil \frac{\lceil m(G) \rceil}{2} \right\rceil = \left\lceil \frac{m(G)}{2} \right\rceil < m(F).
\end{equation} 
In particular, the density of $R$ is strictly smaller than $m(F)$ thus $R \ncong F$.

Let now $F$ and $G$ be graphs such that  condition $(ii)$ holds. As in the previous case, let us fix an orientation of the edges of $G$ such that each vertex has out-degree at most $\lceil m(G) \rceil$. Note that in every (oriented) subgraph $G' \subseteq G$ there exists a vertex with out-degree strictly smaller than $\lceil m(G) \rceil$ as otherwise we would have that the density of such a subgraph is $\lceil m(G) \rceil > m(G)$. Therefore, we can greedily arrange the vertices of $G$ into a sequence $v_1, \ldots, v_n$ such that $N_i := N^+(v_i) \cap \{v_{i+1}, \ldots, v_{n}\}$ is of size at most $\lceil m(G) \rceil-1$. Now the coloring strategy is as follows: for each vertex $v_i$, first arbitrarily pair all the out-edges corresponding to $N_i$ and then all other out-edges incident to $v_i$ and assign a new color to each pair. It remains to prove that there are no rainbow $F$-copies under such  coloring.

Consider some rainbow subgraph $R \subseteq G$. It follows from the pairing strategy that every vertex in $R$ has out-degree at most $\lceil \lceil m(G) \rceil / 2\rceil = \lceil m(G)/ 2 \rceil = m(F)$. Now consider
the vertex $v_i \in V(R)$ with the smallest index $i$ among all the vertices in $R$. Observe that all out-neighbours of $v_i$ in $R$ have index larger than $i$. Since $|N_i| \le \lceil m(G) \rceil - 1$ the pairing strategy ensures that the out-degree of $v_i$ in $R$ is at most
$$ \left\lceil \frac{\lceil m(G) \rceil - 1}{2} \right\rceil < \left\lceil \frac{\lceil m(G) \rceil}{2} \right\rceil = m(F), $$
where the strict inequality follows from the fact that $\lceil m(G) \rceil$ is odd. Thus all the vertices in $R$ have out-degree at most $m(F)$ and at least one vertex has out-degree strictly smaller than $m(F)$. Therefore, the density of any rainbow subgraph $R$ is strictly smaller than $m(F)$ hence there is no rainbow $F$-copy in $G$.
\end{proof}

It remains to cover the case $F = K_4$.

\begin{lemma} \label{lemma:k4_2bnd}
Let $G$ be a graph such that $m(G) \le m_2(K_4) = 2.5$. Then $G \naramcolk[2] K_4$.
\end{lemma}

\begin{proof}
Let us assume towards a contradiction that there exists a graph $G$ on $n$ vertices with $m(G) \le 2.5$ and such that $G \aramcolk[2] K_4$. Without loss of generality let $G$ be a minimal such graph with respect to the number of vertices.

First, observe that $G$ does not contain a vertex $v \in V(G)$ with $\deg(v) < 5$. Otherwise, by taking any  maximal coloring of edges incident to $v$, we have that no rainbow $K_4$-copy can contain $v$. Since it follows from the minimality of $G$ that there is no rainbow $K_4$-copy in $G \setminus \{v\}$ we get that $G$ does not contain a rainbow $F$-copy, thus a contradiction. Therefore, $\delta(G) \ge 5$ and since
$$\sum_{v \in V(G)} \deg(v) \le m(G) \cdot 2n \le m_2(K_4) \cdot 2n = 5 n$$
it follows that $G$ is $5$-regular. Observe that $G \ncong K_6$, as the coloring (see Figure~\ref{fig:k6})
\begin{align*}
& (\{v_1, v_2\}, \{v_1, v_3\}), (\{v_1, v_4\}, \{v_1, v_5\}), \\
& (\{v_1, v_6\}, \{v_5, v_6\}), (\{v_2, v_4\}, \{v_2, v_6\}),\\
& (\{v_3, v_4\}, \{v_3, v_6\}), (\{v_3, v_5\}, \{v_2, v_5\}), (\{v_4, v_5\}, \{v_4, v_6\})
\end{align*}
shows that $K_6 \naramcolk[2] K_4$.

\begin{figure}[h]
\center
\includegraphics[trim = 30mm 60mm 20mm 45mm, clip, scale = 0.2]{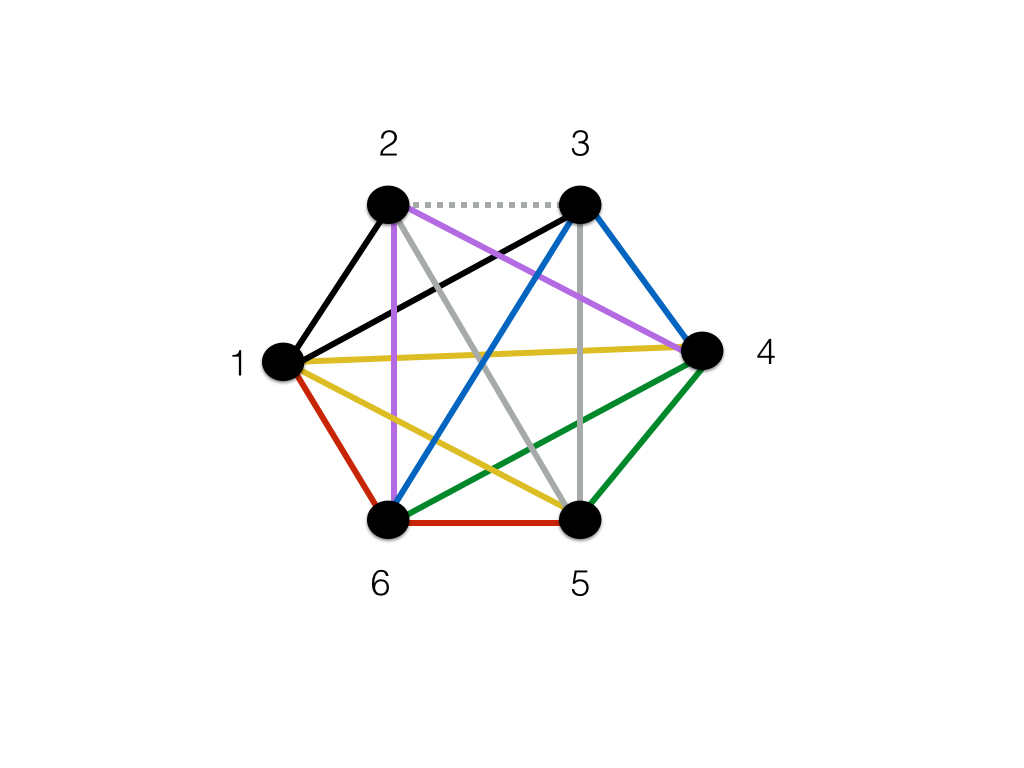}
\caption{A coloring of $K_6$ without a rainbow $K_4$-copy.}\label{fig:k6}
\end{figure}

Let now $v \in G$ be an arbitrary vertex and $N(v) = \{w_1, \ldots, w_5\}$. Assume first that 
$\delta(G[N(v)])\leq 2$ and w.l.o.g.\ let $w_1, w_2$ and $w_3$ be the vertices such that 
$\{w_1, w_2\}, \{w_1, w_3\} \notin E(G)$. Consider the following coloring of the edges incident to $v$:
$$
(\{v, w_1\} ), (\{v, w_2\}, \{v, w_3\}), (\{v, w_4\}, \{v, w_5\}).
$$
Now any possible rainbow $K_4$-copy which contains the vertex $v$ must also contain the vertex $w_1$ and one of the vertices from $\{w_2, w_3\}$. However, that is not possible as $w_1$ is not connected to any of $w_2$ and $w_3$. On the other hand, by the minimality of $G$ no rainbow $K_4$-copy lies completely in $G \setminus \{v\}$. Thus $G$ contains no rainbow $K_4$-copy, which is a contradiction with the choice of $G$.

Therefore, we can assume that $\delta(G[N(v)])\geq 3$. As $G$ is $5$-regular, this implies that every vertex $w_i\in N(v)$ has at most one neighbor in $V(G) \setminus (N(v)\cup\{v\})$. Thus, any $K_4$-copy
that contains a vertex from $N(v)\cup\{v\}$ can contain at most one vertex from $V(G) \setminus (N(v)\cup\{v\})$, which in turn implies that any such clique has to contain three vertices in $N(v)$. However, one easily checks that this can only be if one of the remaining two vertices in $N(v)$ has degree at most two within $G[N(v)]$, which we have already excluded. Thus, there exists no $K_4$-copy which contains a vertex in $N(v)$ and a vertex in $V(G) \setminus (N(v)\cup\{v\})$. We can thus color $G[N(v) \cup \{v\}]$ and $G[V(G) \setminus (N(v) \cup \{v\})]$ separately and by the minimality of $G$ a coloring without rainbow $K_4$-copy exists for both these graphs. This concludes the proof of the lemma.
\end{proof}

We are now ready to combine the previous claims. 

\begin{proof}[Proof of Lemma \ref{lemma:nbounded_general}]
Let us first consider a graph $F$ on four vertices. There exist only two such graphs that are strictly $2$-balanced: $C_4$ and $K_4$. Therefore, if $F$ is a graph on four vertices then $F \cong K_4$ and the conclusion of the lemma follows from Lemma \ref{lemma:k4_2bnd}. For the rest of the proof we assume that $F$ contains at least $5$ vertices and since $F$ is a strictly $2$-balanced graph we have $\delta(F) \ge 2$.

Let $m_2(F) = k + x$ for some $k \in \mathbb{N}$, $k\ge 1$ and $x \in [0, 1)$. Observe that $\delta(F) > m_2(F)$ as otherwise removing a vertex with degree at most $m_2(F)$ would result in a graph with the same or larger $2$-density, which cannot be since $F$ is strictly $2$-balanced. Thus $\delta(F) \ge k + 1$. If $x < 1/2$ then $m(G) < k + 1/2 = (k+1) - 1/2$ and the lemma follows from Claim \ref{claim:min_degree_col}. So we may assume in the following that $x\ge 1/2$.

One easily checks that
$$\frac{3}{4} v(F)^2 - v(F) > \binom{v(F)}{2} \ge e(F)$$
(as $v(F)\ge 5$) and thus
$$ \frac{e(F)}{v(F)} + 3/2 > \frac{e(F) - 1}{v(F) - 2}. $$
As $x \ge 1/2$ this implies $m(F) > m_2(F) - 3/2 \ge k - 1$. For $k\ge 3$ we therefore have
$$\lceil m(G) / 2 \rceil \le \lceil (k+1)/2 \rceil \stackrel{(k\ge 3)}{\leq} k - 1 < m(F),$$
and $G \naramcolk[2] F$ follows from Claim \ref{claim:orientation}. So from now on we may assume that $x\ge 1/2$ and $k\in\{1,2\}$.

Furthermore, if $F$ contains two adjacent vertices of degree $\delta(F)$ then from the fact that $F$ is strictly $2$-balanced and $v(F) \ge 5$ we have
$$ \frac{e(F) - 1 - (2\delta(F) - 1)}{v(F) - 2 - 2} < \frac{e(F) - 1}{v(F) - 2} $$
and so $(2\delta(F) - 1)/2 > m_2(F)\ge k+1/2$. Therefore, either $\delta(F) \ge k + 2$ or $\delta(F) = k + 1$ and $F$ does not contain two adjacent vertices of degree $\delta(F)$. In the first case we trivially have $m(G) \le m_2(F)< k + 1 < k + 2 - 1/2$ and the lemma follows again from Claim \ref{claim:min_degree_col}. In the latter case, if we additionally assume that $x < 5/7$ then
$$\delta(F) - 2/7 \ge k + 5/7 > k + x = m_2(F)$$
and the lemma follows from Claim \ref{claim:2_deg_adj}. Thus we may assume from now on that $x\ge 5/7$ and $k\in\{1,2\}$.

Finally, if $e(F) < (5v(F)^2 - 3v(F))/14$ then 
$$ \frac{e(F)}{v(F)} + \frac{5}{7} > \frac{e(F) - 1}{v(F) - 2},$$
and $x \ge 5/7$ implies that $m(F) > m_2(F) - 5/7 \ge k$. Similarly as before we have
$$ \lceil m(G) / 2 \rceil \le \lceil (k+1)/2 \rceil \le k < m(F), $$
for  $k \in \{1,2\}$ and the lemma follows from Claim \ref{claim:orientation}.

To summarize, we have shown that $G \naramcolk[2] F$ unless the following three conditions hold simultaneously:
\begin{enumerate}[(a)]
\item $x \ge 5/7$,
\item $k \in \{1,2\}$ and
\item $e(F) \ge (5v(F)^2 - 3v(F))/14$.
\end{enumerate}
Let us consider some $F$ such that all three properties apply. Then from (b) and (c) we have
\begin{equation}
3 > m_2(F) \ge \frac{(5v(F)^2 - 3v(F))/14 - 1}{v(F) - 2}. \label{eq:v_F}
\end{equation} 
A simple calculation yields that \eqref{eq:v_F} implies $v(F) < 7$. If $v(F) = 6$ then from (c) we have $e(F) \ge 12$ while from $m_2(F)<3$ we obtain $e(F) \le 12$. 
But then $m(F)\ge 2$ and $\lceil m_2(F) \rceil = 3$ and the lemma follows from the part $(ii)$ of Claim \ref{claim:orientation}. Otherwise, if $v(F) = 5$ then from (c) we have $e(F) \ge 8$ while from $m_2(F)<3$ we obtain  $e(F) \le 9$. However, for $e(F) \in \{8,9\}$ we have $m_2(F) \in \{2 + 1/3, 2 + 2/3\}$ thus $F$ does not satisfy (a). This finishes the proof.
\end{proof}

\begin{proof}[Proof of Lemma \ref{lemma:nbounded_C4}]
Assume towards a contradiction that $G$ is a graph on $n$ vertices such that $m(G) < m_2(C_4) = 3/2$ and $G \aramcolk[2] C_4$. Furthermore, let $G$ be a minimal such graph with respect to the number of vertices. Then
$$\sum_{v \in V(G)} \deg(v) \le 2m(G) n < 3n$$
implies that there exists a vertex $v \in V(G)$ with $\deg(v) \le 2$. Coloring $G\setminus\{u\}$ by the minimality assumption on $G$ and the two edges incident to $v$ with the same (new) color yields a coloring of $G$ with no rainbow $C_4$-copy, contradicting our choice of $G$. 

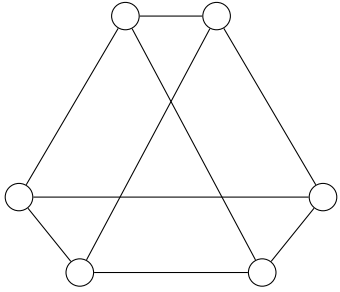
\begin{figure}[h]
\centering
\begin{tikzpicture}[scale = 0.8]
	\node [draw, circle] (0) at (-0.75, 3) {};
	\node [draw, circle] (1) at (0.75, 3) {};
	\node [draw, circle] (2) at (-2.5, 0) {};
	\node [draw, circle] (3) at (-1.5, -1.25) {};
	\node [draw, circle] (4) at (1.5, -1.25) {};
	\node [draw, circle] (5) at (2.5, 0) {};
	\draw (0) to (1);
	\draw (1) to (5);
	\draw (5) to (4);
	\draw (4) to (3);
	\draw (3) to (2);
	\draw (2) to (0);
	\draw (0) to (4);
	\draw (1) to (3);
	\draw (2) to (5); 
  \end{tikzpicture}
\caption{A counter-example for the  case $F = C_4$.} \label{fig:C6}
\end{figure}

For the second part of the lemma, consider the graph $C_6^{3+}$ given in  Figure~\ref{fig:C6}. It is easy to see that $m(C_6^{3+}) = 3/2$. Furthermore, it follows from the fact that the graph is $3$-regular that every pair of edges is contained  in at most two $C_4$-copies. As there are $9$ edges in $C_6^{3+}$, in every $2$-bounded coloring there are at most $4$ pairs of edges which are colored the same. It now follows from the previous observation that every such pair of edges can prevent at most two $C_4$-copies from being rainbow. However, there are $9$ copies of $C_4$, thus at least one copy has to be rainbow. This finishes the proof.
\end{proof}

\subsubsection{Proof of Lemma \ref{lemma:nbounded_hyper}}

We use the following notion of a \emph{link} in a hypergraph.

\begin{definition}[Hypergraph link]
Let $\ell \ge 2$ be an integer and $G$ an $\ell$-graph. Then for a vertex $v \in V(G)$ we define the \emph{link} of $v$ in $G$ to be the $(\ell-1)$-graph $G_v$ induced by the set of edges
$$\{e \setminus \{v\} \; : e \in E(G) \; \text{and} \; v \in e\; \}.$$
Furthermore, define the link of two vertices $v$, $w$ in $G$ to be  the $(\ell-2)$-graph $G_{v,w}$ induced by the set of edges 
$$\{e \setminus \{v,w\} \; : e \in E(G) \; \text{and} \; v,w \in e\; \}.$$
\end{definition}

We make a series of claims towards the proof of Lemma~\ref{lemma:nbounded_hyper}. 
\begin{claim}\label{obs:link_col}
Let $G$ be a vertex minimal $\ell$-graph such that $G\aramcolk[2] K^{(\ell)}_r$. Then
\[
 G_u \;\aramcolk[2]\; K_{r-1}^{(\ell - 1)}.
\]
for every vertex $u$. 
\end{claim}
\begin{proof}
Assume the contrary. Then there exists a $2$-bounded coloring $c_u$ of $G_u$ without a rainbow $K^{(\ell-1)}_{r-1}$-copy. Let $c$ be the partial coloring of $G$  given by
$$ c(e) := c_u(e \setminus \{u\}) $$
for all $e \in E(G)$ with $u \in e$. Then $u$ cannot belong to a rainbow $K_{r}^{(\ell)}$-copy in $G$. As we can also color $G\setminus\{u\}$ without a rainbow $K_{\ell+1}^{(\ell)}$-copy by the the minimality assumption on $G$, this thus contradicts the assumption of the claim $G\aramcolk[2] K^{(\ell)}_r$. 
\end{proof}

\begin{claim}\label{obs:except_1}
Let $G$ be a graph with at most $8$ edges. Then $G \aramcolk[2] K_3$ if and only if $G$ contains a copy of $K_4$. Furthermore, if $G[K] \cong K_4$ for some $K \subseteq V(G)$, then for every $T\in\binom{K}{3}$ there is a $2$-bounded colouring of $G$ with $G[T]$ being the only rainbow $K_3$-copy in $G$.
\end{claim}
\begin{proof}
One easily checks that $K_4 \aramcolk[2] K_3$, thus if $G$ contains $K_4$ then  $G \aramcolk[2] K_3$ as well. In the other direction, let $G$ be a vertex minimal graph with at most $8$ edges without a copy of $K_4$ such that $G \aramcolk[2] K_3$. If $v(G)\ge 6$ then $\delta(G)\le \lfloor16/6\rfloor= 2$, allowing thus a $2$-bounded colouring without a rainbow $K_3$-copy similar to the argument in Lemma~\ref{lemma:nbounded_C4}. Otherwise, for $v(G) \in \{4,5\}$ one easily checks that $G \naramcolk[2] K_3$, thus contradicting the choice of $G$.

For the furthermore-part, observe that if $G[K] \cong K_4$ for some $K \subseteq V(G)$, then $G$ contains at most two additional edges $e_1, e_2 \notin G[K]$. Let us assume that $K = \{v_1, v_2, v_3, v_4\}$ and, without loss of generality, $T = \{v_1, v_2, v_3\}$. 
Then the following $2$-bounded colouring has the required property:
\[
(e_1, e_2), (\{v_1, v_2\},\{v_1,v_3\}), (\{v_1,v_4\},\{v_4,v_2\}), (\{v_2,v_3\},\{v_3,v_4\}).
\]
\end{proof}

\begin{claim}\label{obs:except_2}
Let $G$ be a $3$-graph with at most $16$ edges and no isolated vertices.  
Then $G \aramcolk[2] K^{(3)}_4$ if and only if $G$ is isomorphic to a $3$-graph which consists of two copies of $K^{(3)}_5$ that share $4$ vertices. 
\end{claim}
\begin{proof}
 If $G$ consists of two copies of $K^{(3)}_5$ that share $4$ vertices, then $v(G)=6$, 
 $e(G)=16$ and $G$ contains $9$ copies of $K^{(3)}_4$. Since any pair of edges coloured the 
 same can prevent at most one rainbow $K^{(3)}_4$-copy and in any $2$-bounded colouring of $G$ there are at most $8$ different pairs of edges which 
 are coloured the same,  it follows that one copy of $K^{(3)}_4$ will always be rainbow.  
 
In the other direction, let $G$ be a vertex-minimal $3$-graph on the vertex set $\{v_1, \ldots, v_n\}$ with at most $16$ edges such that $G \aramcolk[2] K^{(3)}_4$.  If $n \le 5$ then $G \subseteq K_5^{(3)}$ and the following $2$-bounded colouring of $K_5^{(3)}$ gives a contradiction with the choice of $G$:
\begin{align*}
&(\{v_1,v_2,v_5\},\{v_1,v_3,v_5\}), (\{v_1,v_4,v_5\},\{v_3,v_4,v_5\}), (\{v_2,v_4,v_5\},\{v_1,v_2,v_4\}),\\
& (\{v_2,v_3,v_4\},\{v_2,v_3,v_5\}), (\{v_1,v_2,v_3\},\{v_1,v_3,v_4\}).
\end{align*}
Therefore, from now on we can assume that $n \ge 6$. Next, let us assume towards the contradiction that $G$ does not contain a $K_5^{(3)}$-copy.  Let $v_i$ be a vertex  of minimum degree which is at most  $\lfloor16\cdot 3/n \rfloor\le 48/6= 8$. Then by Claims~\ref{obs:link_col} and~\ref{obs:except_1} and the minimality assumption on $G$, the link of $v_i$ contains a $K_4$-copy. 
As $G$ does not contain a $K^{(3)}_5$-copy, we know that there exists a $3$-subset $T$ of the vertices of a $K_4$-copy in $G_v$ such that $T \notin E(G)$.  
 Since $e(G_v)\le 8$, Claim~\ref{obs:except_1} asserts the existence of  a $2$-bounded colouring $c_v$ of $G_v$ such that the only rainbow $K_3$-copy is induced by $T$. 
By the minimality of $G$ we can colour $G\setminus\{v\}$ without a rainbow $K^{(3)}_4$-copy. We then extend such colouring to $G$ by using $c_v$ to color the edges containing $v$, 
thus obtaining a colouring without a rainbow $K^{(3)}_4$-copy. 
This is a contradiction with the choice of $G$.

Without loss of generality, we may now assume $G[\{v_1, \ldots, v_5\}] \cong K_5^{(3)}$. Then from $e(G)\le16$ and $e(K^{(3)}_5)=10$ it follows that $\deg_G(v_i) \le 6$ for every $v_i \in \{v_6, \ldots, v_n\}$. By Claims~\ref{obs:link_col} and~\ref{obs:except_1}, we know that the link of every vertex has to contain a copy of $K_4$. Thus $G_{v_6} \cong K_4$ and every edge of $G$ has to either contain $v_6$ or belong to $G[\{v_1, \ldots v_5\}]$. This is only possible if $v(G) = 6$ and so $G$ is isomorphic to  two copies of $K^{(3)}_5$ that share $4$ vertices. 
\end{proof}

We combine the previous claims to derive the following lemma, which we then use as a base for the induction in the proof of Lemma \ref{lemma:nbounded_hyper}.

\begin{lemma}\label{lem:four_critical}
If $G$ is a $4$-graph with $m(G) \le 4$ then $G\naramcolk[2] K^{(4)}_5$.
\end{lemma}
\begin{proof}
Suppose the claim is false and let $G$ be a vertex-minimal $4$-graph with $m(G)\le 4$ and $G \aramcolk[2] K_5^{(4)}$. Since $4\ge m(G)\ge \sum_{x\in V(G)} \deg(x)/(4 v(G))$, it follows from the minimality of $G$ and Claims~\ref{obs:link_col} and~\ref{obs:except_2} that for all $x \in V(G)$ we have $\deg(x)=16$ and the link $G_x$ is isomorphic to two copies of $K^{(3)}_5$ sharing $4$ vertices. Consider any vertex $x\in V(G)$ and let two copies of $K^{(3)}_5$ in $G_x$ be on the vertex sets $\{a_1,b_1,b_2, b_3, b_4\}$ and $\{a_2,b_1,b_2,b_3, b_4\}$. Note that $\{x, a_1, a_2, b_i\} \notin E(G)$ for every $b_i \in \{b_1, b_2, b_3, b_4\}$.

Next, we consider the link $G_{a_1}$. Then $\{b_1, b_2, b_3, b_4, x\} \in V(G_{a_1})$ and let $a'$ be the remaining vertex. If $a' \neq a_2$ then there exists $b_i$, say $b_1$, such that $\{b_2, b_3, b_4, a_1, a_2, a', x\} \in V(G_{b_1})$, which is not possible. Applying the same argument to $G_{a_2}$, we have
$$ V(G_{a_1}) = \{b_1, b_2, b_3, b_4, x, a_2\} \quad \text{and} \quad V(G_{a_2}) = \{b_1, b_2, b_3, b_4, x, a_1\}. $$
It follows now from $\{x, a_1, a_2, b_i\} \notin E(G)$ that $\{b_1, b_2, b_3, b_4\}$ induces a $K_4^{(3)}$-copy in $G_{a_1}$ and $G_{a_2}$, and furthermore $\{b_1, b_j, a_1, a_2\} \in E(G)$ for every $b_j \in \{b_2, b_3, b_4\}$. This implies $\deg(b_1) \ge 18$, thus a contradiction.
\end{proof}

We are now ready to prove Lemma \ref{lemma:nbounded_hyper}. We split the proof into two parts. First, we consider cliques of the type $K_{\ell + 1}^{(\ell)}$.

\begin{proof}[Proof of Lemma \ref{lemma:nbounded_hyper} -- small cliques $K^{(\ell)}_{\ell+1}$,  $\ell\ge 4$ ]
We prove the assertion by induction on $\ell$. The case $\ell=4$ follows from Lemma~\ref{lem:four_critical} as 
$m_4(K^{(4)}_5)=4$. Next, let $\ell>4$ and  assume that the claim holds for $K^{(\ell-1)}_{\ell}$. Let us assume towards the contradiction that there exists an $\ell$-graph $G$ with $m(G)\le m_{\ell}(K^{(\ell)}_{\ell+1})=\ell$ such that $G \aramcolk[2] K_{\ell+1}^{(\ell)}$. Furthermore, let $G$ be a vertex-minimal such $\ell$-graph. 
Claim~\ref{obs:link_col} implies 
\[
 G_u \aramcolk[2] K^{(\ell-1)}_{\ell}
\]
for every vertex $u\in V(G)$. 
By the induction hypothesis we must have
$$m(G_u)>m_{\ell-1}(K^{(\ell-1)}_{\ell})=\ell-1.$$
Consider some $S\subseteq V(G_u)$ such that $m(G_u) = e(G_u[S]) / |S|$. Note that $|S|\ge \ell+2$ as otherwise $e(G_u[S]) \le \binom{\ell+1}{\ell-1} = \binom{\ell+1}{2}$ and thus $m(G_u) \le \ell/2<\ell-1$, contradicting our assumption. Hence, $e(G_u) \ge e(G_u[S]) = m(G_u) \cdot |S| \ge(\ell-1) (\ell+2) > \ell^2$. On the other hand, a vertex $u$ of minimum degree satisfies
\begin{equation}\label{mindegree}
e(G_u)\le \ell\cdot m(G)\le\ell \cdot m_{\ell}(K^{(\ell)}_{\ell+1})=\ell^2,
\end{equation}
yielding the desired contradiction.
\end{proof}

\begin{proof}[Proof of Lemma \ref{lemma:nbounded_hyper} -- large cliques $K^{(\ell)}_{r}$, $r\ge \ell+2$ ]
We prove the lemma by induction on $\ell$. For $\ell = 2$ the claim follows from Lemma \ref{lemma:nbounded_general}. Let now $\ell > 2$ and assume that the claim holds for all $K_{r}^{(\ell-1)}$ with $r \ge \ell + 2$.

Let us assume towards a contradiction that there exists some $r \ge \ell + 2$ and an $\ell$-graph $G$ with $m(G) \le m_\ell(K_r^{\ell})$ such that $G \aramcolk[2] K_r^{(\ell)}$. Furthermore, we assume that $G$ is a minimal such $\ell$-graph with respect to the number of vertices. We show that then
\begin{equation}
\delta(G) > (r + 1) \cdot m_{\ell - 1}(K_{r-1}^{(\ell - 1)}). \label{eq:hyper_min_degree}
\end{equation}
Assuming that equation \eqref{eq:hyper_min_degree} holds, we can lower bound $m(G)$ as follows,
\begin{align}
e(G) / v(G) &=  \frac{ \sum_{v \in V(G)}\deg(v)}{v(G) \cdot \ell} > \frac{r+1}{\ell} \cdot m_{\ell - 1}(K_{r-1}^{(\ell - 1)}) = \frac{(r+1) \cdot (\binom{r-1}{\ell - 1} - 1)}{ \ell (r - 1 - \ell + 1)} \nonumber \\
&= \frac{(r + 1) \binom{r-1}{\ell - 1}}{\ell (r - \ell)} - \frac{r + 1}{\ell (r - \ell)} \stackrel{(r\ge\ell+2)} \ge \frac{(r + 1) \cdot \tfrac{\ell}{r} \binom{r}{\ell}}{\ell (r - \ell)} - \frac{r + 1}{r} \nonumber \\
&> \frac{r + 1}{r} \cdot m_\ell({K_r^{(\ell)}}) - \frac{r + 1}{r} = m_\ell(K_r^{(\ell)}) + \frac{m_\ell(K_r^{(\ell)}) - (r + 1)}{r}. \label{eq:hyper_density_est}
\end{align}
On the other hand, for $r \geq \ell+3$ and since $\ell \ge 3$ we have
$$ m_\ell(K_r^{(\ell)}) = \frac{\binom{r}{\ell} - 1}{r - \ell} \ge \frac{\binom{r}{3} - 1}{r - 3} \ge r + 1. $$
Furthermore, for $r=\ell+2\ge 6$ we have $m_\ell(K_r^{(\ell)})\ge r+1$ as well. 
 Together with \eqref{eq:hyper_density_est}  this implies $m(G) > m_\ell(K_r^{(\ell)})$ for $r\ge \ell+2$ but $(r,\ell)\neq(5,3)$, which contradicts our choice of $G$ in this case. It remains to consider the cases $r = 5$ and $\ell = 3$.   
 One easily checks that in this case
$$ m(G) \stackrel{\eqref{eq:hyper_density_est}}> \frac{r + 1}{\ell} \cdot m_{\ell-1}(K_{r-1}^{(\ell - 1)}) \ge m_\ell(K_r^{(\ell)}), $$
again contradicting the assumption on $G$. Therefore, no such $G$ exists and the claim follows.

It remains to prove equation \eqref{eq:hyper_min_degree}. Consider some vertex $u \in V(G)$ of minimum degree. Similarly to the case of $K^{(\ell)}_{\ell+1}$ cliques, the minimality of $G$ implies that 
\begin{equation}
 G_u \;\aramcolk[2]\; K_{r-1}^{(\ell - 1)}. \label{eq:link_G_v}
\end{equation}
With~\eqref{eq:link_G_v} it follows from the induction assumption that
\begin{equation}
m(G_u) > m_{\ell - 1}(K_{r-1}^{(\ell - 1)}). \label{eq:m_G_u}
\end{equation}
One easily checks that
\begin{align*}
m(K_{r}^{(\ell - 1)}) = \frac 1 r \binom{r}{\ell - 1}  = \frac{1}{r} \cdot \frac{r}{r - \ell + 1} \binom{r-1}{\ell - 1} < \frac{\binom{r-1}{\ell - 1} - 1}{r - \ell} = m_\ell(K_{r-1}^{(\ell - 1)}).
\end{align*}
Together with \eqref{eq:m_G_u} this implies that the densest subgraph of $G_u$ has to be a graph on at least $r+1$ vertices. Thus, we get from \eqref{eq:m_G_u} that
$e(G_u) > (r+1) \cdot m_{\ell-1}(K_{r-1}^{(\ell - 1)})$ and as $\delta(G) \geq e(G_u)$ this concludes the proof of \eqref{eq:hyper_min_degree}.
\end{proof}

\subsection{The Ramsey problem for hypergraph cliques} \label{sec:ramsey_cliques}

As a last application of our method we give a proof of Theorem \ref{thm:ramsey_cliques}. 

\begin{proof}[Proof of Theorem \ref{thm:ramsey_cliques}]
Observe that if a hypergraph $G$ is not $2$-bounded anti-Ramsey for $F$ then it is also not Ramsey for $F$. Indeed, consider some $2$-bounded colouring of $G$ without a rainbow copy of $F$. As each colour occurs at most twice, we can colour one edge red and the other one blue. Now observe that any monochromatic subgraph in this colouring  corresponds to a rainbow subgraph in the original colouring. Thus, no monochromatic copy of $F$ appears. As an immediate consequence of Theorem \ref{thm:2_bnd} we get the $0$-statement of Theorem \ref{thm:ramsey_cliques} for all $\ell$-graphs which are cliques of size at least $\ell + 1$ with the exception of the (hyper)graphs $K_3$ and $K^{(3)}_4$. The case of $K_3$ was already shown in Theorem~\ref{thm:rr}, thus it remains to consider $K^{(3)}_4$. 

Note that  Algorithm \ref{algo:ar_2bnd}, with line $4$ changed such that it assigns red colour to $e_1$ and blue to $e_2$, provides a $2$-colouring of the hypergraph $G$. As the analysis and the correctness of the algorithm remains the same as in the proof of Theorem \ref{thm:2_bnd}, it suffices to show that if $m(G') \le m_3(K^{(3)}_4)=3$ then $G'  \nramcolk[2] K^{(3)}_4$.



Let $G$ be a vertex minimal graph with $G \ramcolk[2] K^{(3)}_4$ and let
 $u \in V(G)$ be a vertex  of minimum degree. Claim~\ref{obs:link_col}  yields $G_u  \ramcolk[2] K_3$. However, $\deg(u)\le 3\cdot m(G)\le 3\cdot m_{3}(K^{(3)}_4)=9$ and it is easy to see that any graph with less than $15$ edges is not Ramsey for $K_3$ and two colors, see e.g.~\cite{EFRS78}.
\end{proof}

In a forthcoming paper~\cite{GNPSST} we extend Theorem~\ref{thm:ramsey_cliques} to various classes of $\ell$-graphs other than cliques. Furthermore we find examples of $\ell$-graphs $F$, where the threshold $p$ for $\Hknp\ramcolk[k] F$ is neither determined by $m_{\ell}(F)$ nor by a density  $m(G)$ of some obstruction $\ell$-graph $G$, 
but rather exhibits some asymmetric behaviour. 

\bibliographystyle{abbrv}
\bibliography{ramsey}

%
%

\end{document}